\documentclass{imsd2026_preprint}

\begin{document}

\begin{center}
  \fontsize{12}{16}{\bf
    1-Lipschitz Neural Networks on Hadamard Manifolds
    }
\end{center}

\begin{center}
\normalsize{
  \bf{
    Davide Murari$^{*}$,
    Marta Ghirardelli$^\dag$,
    Ben Adcock$^\#$,
    Elena Celledoni$^\dag$
    ,
    Brynjulf Owren$^\dag$
    Carola-Bibiane Sch\"onlieb$^{*}$
    }
}
\end{center}
\begin{center}
  \begin{tabular}{c}
    Department of Applied Mathematics and Theoretical Physics, \\
    University of Cambridge \\
    {[dm2011, cbs31]@cam.ac.uk}
  \end{tabular}

  \vspace{0.5ex}

  \begin{tabular}{c}
    $^\dag$ Department of Mathematical Sciences, \\
    Norwegian University of Science and Technology (NTNU), \\
    {[marta.ghirardelli, elena.celledoni, brynjulf.owren]@ntnu.no} \\
  \end{tabular}
  
  \vspace{0.5ex} 

  \begin{tabular}{c}
    Department of Mathematics, \\
    Simon Fraser University, \\
    {adcockb@sfu.ca}     \\	
  \end{tabular} \\

\end{center}

\section*{ABSTRACT}
Controlling the Lipschitz constant of a neural network is a standard way to promote robustness
and stability. Most existing constraining strategies are designed for Euclidean spaces. In this work, we construct and analyze a class of 1-Lipschitz neural networks on Hadamard manifolds. Our layers are of gradient-descent type, $1$-Lipschitz, and quasi-$\alpha$-firmly nonexpansive. The core building blocks of the proposed architecture are Busemann functions, and we exploit the properties of Busemann gradient flows to design $1$-Lipschitz geometry-preserving layers. We provide explicit constructions and examples for hyperbolic manifolds and the manifold of symmetric positive definite (SPD) matrices. We test the proposed architecture in two numerical experiments: robust classification on the Poincar\'e disk and masked-Wishart covariance reconstruction. On the Poincar\'e disk, the proposed networks yield robust classifiers under hyperbolic
perturbations. On the SPD manifold, we train SPD-valued denoisers and adopt them as a
Plug-and-Play prior for a masked-Wishart covariance reconstruction problem. We show improved results from the nonexpansive denoiser over static, data-only, and Log-Euclidean denoising baselines, and empirically test its
convergence properties.

\textbf{Keywords: }
1-Lipschitz neural networks,
Hadamard manifolds,
Busemann functions,
geometric deep learning,
hyperbolic neural networks,
Plug-and-Play

\textbf{MSCcodes: }
68T07, 53C20, 47H09, 47H10


\section{Introduction}
Controlling the sensitivity of a neural network is important in several settings where stability is a modeling requirement. This is the case in adversarially-robust classification, in Wasserstein-based generative modeling, and in learned iterative methods for inverse problems, where the network often appears inside a fixed-point iteration and must interact well with the surrounding algorithmic structure \cite{cisse2017parseval,miyato2018spectral,anil2019sorting,combettes2020lipschitz,sherry24dsn}. In Euclidean spaces, this has led to an extensive literature on architectures with controlled Lipschitz constant. This constraint is usually enforced through spectral normalization, orthogonality constraints, monotone and gradient-type residual blocks \cite{cisse2017parseval,miyato2018spectral,qian2019l2,anil2019sorting,combettes2020lipschitz,meunier2022dynamical,sherry24dsn}. These constructions show that stability can be encoded structurally, rather than checked a posteriori.

Many learning problems are not naturally posed in flat spaces. Data may satisfy intrinsic constraints, or admit geometric representations that a curved latent space can capture efficiently. This is one of the main themes of geometric deep learning, where the goal is to design models that respect the geometry of the domain rather than forcing it into Euclidean coordinates \cite{bronstein2021geometric}. Hyperbolic representation learning is one of the most prominent examples of this paradigm. Negatively curved spaces are well-suited to hierarchical or tree-like data and have therefore led to a rich literature on hyperbolic embeddings and hyperbolic neural networks \cite{nickel2017poincare,sala2018representation,ganea2018hyperbolic,peng2021hyperbolic}. More broadly, manifold-valued architectures already appear in intrinsic convolutions on Riemannian manifolds, networks on the manifold of symmetric positive definite (SPD) matrices, and constructions on noncompact symmetric spaces \cite{masci2016geodesic,huang2017riemannian,sonoda2022symmetric,nguyen2023matrix}. Most of this literature, however, focuses on expressive geometry-aware layers rather than a systematic stability theory.

This matters in concrete applications. Manifold-valued data arise in directional statistics and phase-valued imaging on circles and spheres \cite{mardia2000directional,bergmann2019recent}, in rotation-valued signals and images such as electron backscatter diffraction data on $\mathrm{SO}(3)$ \cite{bachmann2011grain,bergmann2019recent}, and in covariance-based vision where SPD descriptors are used for recognition and classification \cite{tuzel2006region,harandi2014manifold}. A particularly relevant example is diffusion tensor imaging. Here, each voxel is represented by an SPD matrix encoding anisotropic diffusion \cite{weinmann2014total,gorokh2016dti}. These settings call for intrinsic neural layers that respect the geometry of the state space while retaining the stability properties to make them useful in robust learning and inverse problems.

Among stability notions, plain $1$-Lipschitz continuity is often only the first step. In inverse problems and learned fixed-point methods, the more relevant classes are averaged, and firmly nonexpansive operators, since these are the ones that naturally appear in proximal algorithms, forward--backward splitting, plug-and-play methods, and equilibrium formulations \cite{moreau1965proximite,bauschke2017correction,combettes2021fixedpoint,ryu2019plug,hurault2022proximal,bai2019deq,baker2023implicit}. The importance of this viewpoint is that it connects stable layers with convergent learned solvers.

The Euclidean architectures most relevant to the present work are those derived from gradient flows. Recent papers have shown that residual blocks of the form
\[
x \mapsto x - \tau \nabla V(x)
\]
can provide a natural route to nonexpansive (i.e., 1-Lipschitz) and averaged networks \cite{meunier2022dynamical,sherry24dsn,celledoni23dsb}. They have also recently been shown to generate universal $1$-Lipschitz scalar function approximators \cite{murari2026approximation}. For these layers, the Lipschitz constraint follows from the convexity and $L$-smoothness of the potential. The same formalism extends beyond $\mathbb{R}^n$. On a Riemannian manifold, one can replace the Euclidean gradient by the Riemannian one, obtaining the map
\[
T_\tau(x)=\exp_x\bigl(-\tau \grad V(x)\bigr).
\]
Whenever the manifold geometry supports a workable convexity theory, this gives a principled mechanism for building geometry-preserving stable layers.

This is where Hadamard manifolds become the natural setting. These manifolds are complete, simply connected, and nonpositively curved. They therefore combine a broad geometric scope with the global properties needed for analysis. Geodesics are unique, and geodesic convexity admits a global theory closely related to convex analysis in Hilbert spaces \cite{bridson2013metric,bacak2014convex,boumal2023introduction,monotoneFieldsHadamard}. The class includes the standard hyperbolic models used in representation learning, as well as other geometries of independent interest, including the SPD manifolds. Just as importantly, the fixed-point theory survives in this setting. Averagedness must be reformulated through geodesic interpolation, leading to natural extensions of firmly nonexpansive and $\alpha$-firmly nonexpansive mappings on Hadamard spaces \cite{ariza2014firmly,ariza2015asymptotic,berdellima2020notion,berdellima2022alpha}. These results provide the right framework for studying manifold-valued layers.

To the best of our knowledge, however, there is no prior work whose main objective is the design and analysis of stable neural networks on manifolds. Existing papers on manifold-valued networks explain how to build intrinsic architectures, while the Euclidean stability literature explains how robustness and convergent dynamics can be obtained from Lipschitz, averaged, or gradient-based constructions. The gap addressed in this paper is precisely between these two threads. We transport Euclidean stability constructions to the Hadamard setting and study whether they remain both analytically useful and practically effective.

For concrete layer design, we focus on gradient-type updates generated by geodesically convex potentials. On some important Hadamard manifolds, Busemann functions provide a particularly natural source of such potentials and hence a convenient modeling tool for intrinsic layers. In hyperbolic learning, they already appear, explicitly or implicitly, in prototype methods, horospherical classification, dimensionality reduction, and sliced-Wasserstein constructions \cite{ghadimi2021hyperbolic,chami2021horopca,fan2023horospherical,bonet2023hypersw,atigh2026granularity}. On the manifold of symmetric positive definite matrices, horoballs and horospheres have also been studied from a geometric viewpoint. Explicit formulae for the associated radial fields, that is, the negative gradients of Busemann functions, have recently become available \cite{fletcher2011horoball,bonet2025sliced,shin2026radial}. What is missing in these lines of work is the use of Busemann functions as building blocks for stable gradient-type layers equipped with nonexpansive or averaged guarantees. This makes them a natural bridge between existing geometric constructions and the framework developed here.

\subsection{Main contributions}
In this paper, we study intrinsic gradient-descent-type layers on Hadamard manifolds and show how to endow them with stability properties familiar from the Euclidean theory. Our main contributions are as follows.
\begin{itemize}
    \item We establish a bridge between 1-Lipschitz Euclidean gradient layers and intrinsic gradient steps on Hadamard manifolds. We propose Busemann gradient-descent layers and prove their nonexpansiveness under a stepsize restriction. We also show they are quasi $\alpha$-firmly nonexpansive.
    \item We construct implementable manifold-valued layers from tractable families of potentials, with particular attention to hyperbolic manifolds and the manifold of SPD matrices where Busemann functions yield explicit, computationally convenient formulae.
    \item We validate the proposed layers on two numerical tasks, namely adversarially robust classification in the Poincar\'e disk and denoising on the manifold of SPD matrices, and compare them with baseline manifold models. The numerical implementation of the networks and experiments can be found at the associated GitHub repository \href{https://github.com/davidemurari/one-lipschitz-hadamard-networks}{https://github.com/davidemurari/one-lipschitz-hadamard-networks}.
\end{itemize}

\subsection{Outline of the paper}
The paper is organized as follows. After the notation subsection, \Cref{sec:background} presents some background material and introduces the property of quasi $\alpha$-firm nonexpansiveness of certain nonexpansive maps. In \Cref{sec:Busemann layers}, we introduce Busemann functions and the proposed layers, whose nonexpansiveness condition is derived in \Cref{sec:Busemann layers nonexpansive}. The implementation of the architecture is discussed in \Cref{sec:implementation}, then numerical experiments are shown in \Cref{sec:numerical experiments}, and \Cref{sec:conclusion} concludes the paper. Deferred proofs are collected in \Cref{appendix: proofs omitted in the main document}.

\subsection{Notation}
Throughout the paper, $(\mathcal{M},g)$ denotes a smooth Hadamard manifold. We write $T_x\mathcal{M}$ for the tangent space at $x\in\mathcal{M}$.

We write $\langle \cdot,\cdot\rangle_x$ and $\|\cdot\|_x$ for the inner product and norm induced by $g$ on $T_x\mathcal{M}$. When the base point is clear from the context, we write $\langle \cdot,\cdot\rangle$ and $\|\cdot\|$.
The length of a differentiable curve $\gamma:[a,b] \subseteq \mathbb{R} \to \M$ is denoted by
\begin{equation*}
    \ell (\gamma) = \int_a^b \| \dot \gamma (t) \| \dd t,
\end{equation*}
while $d(x,y)$ is the Riemannian (geodesic) distance between points $x$ and $y$ on $\M$. $(\M,d)$ is therefore a metric space. The Levi-Civita connection is denoted by $\nabla$, the Riemannian gradient by $\grad$, and the Riemannian Hessian by $\Hess$.

For $x\in\mathcal{M}$, we denote by $\exp_x:T_x\mathcal{M}\to\mathcal{M}$ the Riemannian exponential map and by $\log_x(y)\in T_x\mathcal{M}$ the inverse exponential map at $x$, well defined for every $y\in\mathcal{M}$ since Hadamard manifolds are globally geodesically convex.

For $\alpha\in[0,1]$ and $x,y\in\mathcal{M}$, the notation
\[
x\#_\alpha y = (1-\alpha)x\oplus \alpha y := \exp_x\bigl(\alpha \log_x y\bigr)
\]
denotes the point at proportion $\alpha$ along the geodesic from $x$ to $y$. We write $\Fix(T)$ for the set of fixed points of a map $T:\mathcal{M}\to\mathcal{M}$, i.e., $\Fix(T):=\{x\in\M:\,\,T(x)=x\}$. We denote by $O(n)$ the set of $n\times n$ real orthogonal matrices.

\section{Background definitions and results}
\label{sec:background}
In this section we present some definitions and results that we use in \Cref{sec:Busemann layers} and \Cref{sec:Busemann layers nonexpansive}. We first define Jacobi fields, which are variation fields of one-parameter families of geodesics \cite[Chapter~10]{lee18itr}. Carrying information about the curvature of the space, they provide a useful tool to analyze the behavior of nearby geodesics. In particular, they can be used to obtain a condition under which geodesics do not spread apart; see \Cref{trm:Ttau nonexpansive}. The second part is dedicated to quasi-$\alpha$-firmly nonexpansive maps, and concludes by deriving such a property for nonexpansive Riemannian gradient-descent maps.

\subsection{Jacobi fields}
Let $\gamma : \mathbb{R} \to \M$ be a unit-speed geodesic. A Jacobi field along $\gamma$ is a vector field satisfying the Jacobi equation
\begin{equation*}
    \nabla^2_t J + R(J, \Dot{\gamma}) \Dot{\gamma} = 0,
\end{equation*}
where $\nabla_{t} = \nabla_{ \Dot{\gamma}(t)}$. Any Jacobi field can be decomposed as
\begin{equation*}
    J(t) = J^\parallel(t) + J^\perp (t) = (\alpha + \beta t) \Dot{\gamma}(t) + J^\perp (t)
\end{equation*}
for some coefficients $\alpha, \beta \in \mathbb{R}$, with $J^\parallel(t)$ \textit{parallel} to $\Dot{\gamma}(t)$, and $J^\perp(t)$ \textit{orthogonal} to $\Dot{\gamma}(t)$, i.e.
\begin{equation*}
    \langle J^\perp(t), \Dot{\gamma}(t) \rangle = 0 \qquad \text{for all } t\in\mathbb{R}.
\end{equation*}
We refer the reader to \cite[Chapter~10]{lee18itr} for a more in-depth study.

\subsection{Quasi-$\alpha$-firmly nonexpansive maps}

A map $T:\M\to\M$ is $1$-Lipschitz, i.e., nonexpansive, if $d(T(x),T(y))\leq d(x,y)$ for every $x,y\in\M$. This constraint does not necessarily lead to a convergent fixed point iteration $x_{k+1}=T(x_k)$, $k\geq 0$. Banach's fixed point Theorem ensures that if $T$ satisfies $d(T(x),T(y))<\beta d(x,y)$, for every $x,y\in\M$ and a $\beta\in (0,1)$, then such an iteration converges to the unique fixed point of $T$. However, this contraction constraint can be overly restrictive in practice. In this subsection, we introduce the necessary mathematical framework to define a broader subclass of nonexpansive maps which leads to convergent fixed point iterations. We start by defining the notion of quasi-$\alpha$-firmly nonexpansive map, and introducing some of the properties of this class of functions.

\begin{definition}\cite[Definition~1]{berdellima2022alpha}
    \label{def: quasi alpha-firmly nonexpansive}
    Let $\mathcal{M}$ be a Hadamard manifold and let $\alpha\in(0,1)$. We say $T:\mathcal{M}\to \mathcal{M}$ is quasi-$\alpha$-firmly nonexpansive if $\Fix(T)$ is non empty, and for every $y\in\Fix(T)$ one has
    \begin{equation*}
        d^2(Tx, y) \le d^2(x,y) - \frac{1-\alpha}{\alpha} d^2 (Tx, x)
    \end{equation*}
    for every $x \in \mathcal{M}$.
\end{definition}

The class of quasi-$\alpha$-firmly nonexpansive maps is constrained enough to ensure the convergence of the fixed point iterations, whenever a fixed point exists, as formalized in the following theorem.

\begin{theorem}
    \label{corollary: convergence to fix point for quasi firmly nonexpansive operators}
    \cite[Corollary 6]{berdellima2020notion}
    Let $\M$ be a finite-dimensional Hadamard manifold, and $T:\M\to \M$ a nonexpansive map. If $T$ is quasi-$\alpha$-firmly nonexpansive then the iterates $x_n := T (x_{n-1})$ for $n\in \mathbb{N}$ converge strongly to some $x^*\in\Fix (T)$.
\end{theorem}

Quasi-$\alpha$-firm nonexpansiveness is closed under composition under some assumptions. Closeness under composition is a condition that simplifies the design of constrained neural networks, since it allows working at the layer level, rather than considering the whole architecture. This is also true for the nonexpansiveness property. We leverage this in the network design process presented in \Cref{sec:implementation}.

\begin{proposition}
    \label{corollary: finite composition of firmly non expansive is firmly nonexpansive}
    \cite[Corollary~13]{berdellima2022alpha}
    Let $\M$ be a Hadamard manifold. Let $T_1 : D_1 \to \M$ where $D_1 \subset \M$, and for $j=2,\dots, m$ let $T_j : D_j \to \M$ where $D_j := \{ T_{j-1}(x) : x\in D_{j-1}\}$. If $T_j$ is quasi-$\alpha$-firmly nonexpansive with constant $\alpha_j$ on $D_j$ for $j=1,\dots,m$, and $\Fix(T_m \circ T_{m-1} \circ \dots \circ T_1) \subset D_1$ is nonempty, then the composite map $T:=T_m \circ T_{m-1} \circ \dots \circ T_1 $ is quasi-$\alpha$-firmly nonexpansive on $D_1$ with constant given recursively by
    \begin{equation*}
        \overline{\alpha}_m = \frac{\overline{\xi}_{m-1} + \xi_m}{\overline{\xi}_{m-1} \xi_m + \overline{\xi}_{m-1} + \xi_m} \qquad \mathrm{for } \,\, m \ge 3,
    \end{equation*}
    where $\overline{\xi}_j := \frac{1-\overline{\alpha}_j}{\overline{\alpha}_j}$ for $j\ge 2$, $\xi_j := \frac{1-\alpha_j}{\alpha_j}$ for $j\ge1$, and $\overline{\alpha}_2 := \frac{\xi_1 + \xi_2}{\xi_1 \xi_2 + \xi_1 + \xi_2}$.
\end{proposition}

On Hadamard manifolds, we can derive a more practical construction of some quasi-$\alpha$-firmly nonexpansive maps. This is formalized in the following theorem.

\begin{theorem}
    \label{trm: alpha averaged operators are quasi alpha FME}
    Let $\mathcal{M}$ be a Hadamard manifold and let $\alpha \in (0,1)$. Let $T_\alpha:\mathcal{M} \to \mathcal{M}$ be defined as
    \begin{equation*}
        T_\alpha (x) = (1-\alpha) x \, \oplus \, \alpha N(x) = \exp_x(\alpha \log_x (N(x))), \qquad x\in\mathcal{M},
    \end{equation*}
    for a nonexpansive map $N:\mathcal{M}\to\mathcal{M}$. If $\Fix(N) \ne \emptyset $, then $T_\alpha$ is quasi-$\alpha$-firmly nonexpansive and $\mathrm{Fix}(T_\alpha)=\mathrm{Fix}(N)$.
\end{theorem}

\begin{proof}
    The proof is similar to 
    the Hilbert space identity for averaged operators, see e.g. \cite[Proposition~2.2]{bacak2014convex}. Let $p\in\Fix(N)$ and $x\in\mathcal{M}$. On Hadamard manifolds, the distance function is uniformly convex \cite[Definition~2.1]{ariza2015asymptotic}, hence
    \begin{equation*}
        d^2(T_\alpha(x), p) \le (1-\alpha) d^2(x,p) + \alpha d^2(N(x), p)
        - \alpha (1-\alpha) d^2(x, N(x)).
    \end{equation*}
    Now $d(N(x),p)=d(N(x),N(p)) \le d(x,p)$ by nonexpansiveness of $N$, and $d(x,T_\alpha(x)) = \alpha d(x, N(x))$ by construction. Therefore,
    \begin{equation*}
        d^2(T_\alpha(x), p) \le d^2(x,p) - \frac{1-\alpha}{\alpha} d^2(x, T_\alpha (x)).
    \end{equation*}
    By the injectivity of the Riemannian exponential and logarithm on Hadamard manifolds, we have that
    \[
    x=T_\alpha(x)\quad \iff \quad \alpha\log_x (N(x)) = 0 \quad \iff N(x)=x,
    \]
    and hence $\Fix(N) = \Fix(T_\alpha)$. This concludes the proof.
\end{proof}

The result above allows us to model quasi-$\alpha$-firmly nonexpansive maps by modeling nonexpansive ones. This is why we focus on building 1-Lipschitz neural networks. Such a construction will be evident in the SPD-inverse problem of \Cref{subsec:spd-inverse}. Additionally, for the gradient-descent-type layers we study in the upcoming sections, we have an even more compatible structure as pointed out by \Cref{trm: Ttau is quasi alpha FNE}.

\begin{theorem}
    \label{trm: Ttau is quasi alpha FNE}
    Let $\mathcal{M}$ be a Hadamard manifold. Let $V:\mathcal{M}\to \mathbb{R}$ be of class $\C^1$, and consider the Riemannian gradient descent step
    \begin{equation*}
        T_\tau(x) = \exp_x(-\tau \grad V(x)),\qquad T_\tau:\M\to \M.
    \end{equation*}
    Assume $T_\tau$ is nonexpansive for $\tau \in [0, \bar \tau]$ and $\Fix(T_\tau) \ne \emptyset$. Then, for every $\tau \in (0,\bar \tau )$, $T_\tau$ is quasi-$\alpha$-firmly nonexpansive for any $ \alpha \in [\tau/ \bar \tau, 1 )$.
\end{theorem}
\begin{proof}
Fix $\tau\in(0,\bar \tau]$ and $\alpha\in(0,1)$. Define $G_{\alpha, \tau}(x) := (1-\alpha)x \,\oplus\,\alpha T_\tau(x)$. Since $T_\tau$ is nonexpansive and $\Fix(T_\tau) \ne\emptyset$ by assumption, \Cref{trm: alpha averaged operators are quasi alpha FME} implies that $G_{\alpha,\tau}$ is quasi-$\alpha$-firmly nonexpansive. On the other hand, we also have $G_{\alpha, \tau}(x)=T_{\alpha\tau}(x)$ for every $x\in\M$.
Thus $T_{\alpha\tau}$ is quasi-$\alpha$-firmly nonexpansive for any $\tau\in(0,\bar \tau ]$ and any $\alpha\in(0,1)$. 
Now let $\tau'\in(0, \bar \tau)$ and choose $\alpha\in[\tau'/\bar \tau,1)$. Then
\[
\tau:=\frac{\tau'}{\alpha}\in \left(0,\bar \tau \right].
\]
By the previous argument, $T_{\alpha\tau}=T_{\tau'}$ is quasi-$\alpha$-firmly nonexpansive. As $\tau'\in(0,\bar \tau )$ was arbitrary, renaming $\tau'$ as $\tau$ yields the claim.
\end{proof}

To the best of our knowledge, the literature does not provide a result showing the nonexpansiveness of the gradient-descent map $T_\tau$ for a general class of potentials $V$, such as geodesically convex and $L$-smooth potentials. For this reason, in \Cref{sec:Busemann layers} we derive a more specialized construction of $V$ that allows us to prove the nonexpansiveness of the corresponding Riemannian gradient-descent step. We then use such a result to design our nonexpansive manifold-valued neural networks. The analysis of the more general case remains open, and could expand the range of $1$-Lipschitz manifold-valued networks one can design following our principles.

\section{Busemann gradient-descent layers}
\label{sec:Busemann layers}
In $\mathbb{R}^n$ equipped with the Euclidean metric, one can consider layers of the form
\begin{equation*}
    x \mapsto x - \tau \nabla V(x) = x - \tau A^\top \sigma(Ax + b),
    \qquad
    x,b \in \mathbb{R}^n,\, A\in\mathbb{R}^{n\times n},\,\tau\geq 0,
\end{equation*}
where $\sigma$ is a scalar activation function, $V(x) = 1^\top_n\varphi(A x+b)$, $\varphi(\tau) = \int_0^\tau \sigma(s) \text{d}s$, and $1_n\in\mathbb{R}^n$ is a vector of ones. Assuming $\sigma$ is non-decreasing and $L$-Lipschitz, the above map is 1-Lipschitz and $\alpha$-averaged whenever $\tau \in (0, 2/(L\|A\|_2^2))$, with $\alpha=\tau\|A\|_2^2L/2$ \cite{sherry24dsn}. We recall that $T:\mathbb{R}^n\to\mathbb{R}^n$ is $\alpha$-averaged if $T(x)=(1-\alpha)x+\alpha N(x)$ for a 1-Lipschitz map $N:\mathbb{R}^n\to\mathbb{R}^n$, as considered for the construction in \Cref{trm: alpha averaged operators are quasi alpha FME}.

In this section, we consider Riemannian descent gradient layers on Hadamard manifolds. We propose a parametrization of the potential $V$ based on Busemann functions, thereby obtaining a concrete class of nonexpansive models on Hadamard manifolds for which Busemann functions and their Riemannian gradients can be efficiently implemented. We first provide the basic notions and properties about such a family of functions, and refer the reader to \cite[Chapter~II.8]{bridson2013metric} for a more in-depth study. We then present the Busemann gradient-descent layers and show that they are 1-Lipschitz under a certain condition on the stepsize (\Cref{trm:Ttau nonexpansive}). 

Let $p\in\M$, $\xi\in\partial\M$ a point at infinity and $\gamma : [0,+\infty) \to \mathcal{M}$ the (unique) unit-speed geodesic ray from $p$ to $\xi$. Let $b$ denote the Busemann function associated to $\gamma$ defined by
\begin{equation}
    \label{eq: busemann on manifold}
    b(x) = \lim_{s\to+\infty} (d(\gamma(s),x)-s).
\end{equation}
When we want to remark the associated geodesic ray or point at infinity, we interchangeably use the notations $b_\gamma$ and $b_\xi$.

The existence of the limit in \eqref{eq: busemann on manifold} is guaranteed by \cite[Chapter II.8, Lemma~8.18]{bridson2013metric}. The level sets of $b$ are called horospheres: for $t\ge 0$
\begin{equation}
    \label{eq: horosphere}
    H_{\gamma(t)} := \{x \in \M : b(x) = b(\gamma(t))\}
\end{equation} 
is the horosphere at $\gamma(t)$, a smooth hypersurface orthogonal to $\gamma$ at $\gamma(t)$.

\begin{lemma}
    \label{lemma: properties of Busemann functions on Hadamard manifolds}
    Let $(\mathcal{M}, g)$ be a Hadamard manifold. Every Busemann function $b$ on $\M$ is geodesically convex, 1-Lipschitz, and of class $\C^{2}$. 
\end{lemma}
The above result is quite standard, and we include a proof in \Cref{appendix: proofs omitted in the main document}. \Cref{lemma: phi coincides with exp} below shows that Busemann flows, i.e., Riemannian gradient flows of Busemann functions, generate unit-speed geodesics, often referred to as Busemann geodesics. This property plays a central role in showing \Cref{trm:Ttau nonexpansive}.
\begin{lemma}
    \label{lemma: phi coincides with exp}
    Let $(\mathcal{M}, g)$ be a Hadamard manifold, and $b$ a Busemann function on $\mathcal{M}$. Let $x\in\M$, and let $(\Phi_t)_{t\geq 0}$ denote the Busemann flow given by
    \begin{equation}
        \label{eq: Busemann flow phi}
        \frac{\dd}{\dd t}\Phi_t (x)=-\grad b(\Phi_t (x)), \qquad \Phi_0 (x)=x.
    \end{equation}
    Then
    \begin{equation*}
        \Phi_t(x) = \exp_x(-t \grad b(x))
    \end{equation*}
    for every $t\ge0$, and $(t,x)\mapsto \Phi_t (x)$ is $\C^1$ jointly in $(t,x)$. Moreover, $b(\Phi_t(x)) = b(x) - t$, and
    \begin{equation*}
        \Hess b_{\Phi_t(x)} (\grad b(\Phi_t(x)), \grad b(\Phi_t(x))) = 0.
    \end{equation*}
\end{lemma}

\begin{proof}
    The flow line $t\mapsto \Phi_t(x)$ is a unit-speed geodesic asymptotic to $\xi$, see \cite[Section~2.3]{criscitiello2025horospherically}, with initial data $\Phi_0(x) = x$, $ \dot \Phi_0(x) = -\grad b(x)$. On the other hand, the curve $t\longmapsto \exp_x(-t \grad b(x)) =: \eta(t)$ is by definition the unique geodesic with initial conditions $\eta(0) = x$, $\dot \eta (0)=-\grad b(x)$. By uniqueness of geodesics with prescribed initial data, the two curves must coincide for every $t\ge0$. \Cref{lemma: properties of Busemann functions on Hadamard manifolds} guarantees that $b$ is $\C^2$, hence the flow map $(t,x)\mapsto \Phi_t(x)$ is $\C^1$ jointly in $(t,x)$. \\
    Consider now the derivative of $b(\Phi_t(x))$ with respect to $t$:
    \begin{equation*}
        \frac{\dd}{\dd t}b(\Phi_t(x)) = \langle \grad b(\Phi_t(x)), -\grad b(\Phi_t(x)) \rangle = -\|\grad b(\Phi_t(x))\|^2 = -1.
    \end{equation*}
    Integrating the above yields $b(\Phi_t(x)) = b(x) - t$, while deriving we find
    \begin{align*}
        \frac{\dd^2}{\dd t^2} b(\Phi_t(x)) &= \Hess b_{\Phi_t(x)} (\Dot{\Phi}_t(x), \Dot{\Phi}_t(x)) \\
        &=
        \Hess b_{\Phi_t(x)} (\grad b(\Phi_t(x)), \grad b(\Phi_t(x))) = 0,
    \end{align*}
    i.e., the Hessian of $b$ vanishes in the directions parallel to $\Dot{\gamma}$ (i.e., to $\grad b(x)$).
\end{proof}

We now use the theory of Busemann functions and their gradient flows to design nonexpansive neural networks on Riemannian manifolds. We briefly introduce these layers now, and then study their properties in \Cref{sec:Busemann layers nonexpansive}.

\paragraph{Definition of the Busemann gradient-descent layers} Let $\gamma$ be a geodesic ray on $\M$, $b$ the Busemann function generated by $\gamma$, $\lambda \ge 0$, $\beta\in\mathbb{R}$, $\varphi:\mathbb{R}\to\mathbb{R}$ of class $\mathcal{C}^1$ and consider the potential
\begin{equation}
    \label{eq:potential}
    V(x)=\varphi(\lambda b(x)+\beta), \qquad
    \grad V(x) = \lambda\varphi'(\lambda b(x)+\beta) \grad b(x).
\end{equation}
The Busemann (Riemannian) gradient-descent layer is defined by
\begin{equation}
    \label{eq:Busemann gradient layer}
    T_\tau(x)=\exp_x(-\tau\grad V(x)), \qquad \tau \ge 0.
\end{equation}
It follows from \Cref{lemma: phi coincides with exp} that, for every fixed base point \(x\in \M\),
\begin{equation}
    \label{eq:single-busemann-gradient-step}
    T_\tau(x) = \Phi_{a(b(x))}(x),
\end{equation}
where $a(r):=\tau \lambda\varphi'(\lambda r+\beta)$, $\Phi$ is the Busemann flow as in \eqref{eq: Busemann flow phi}, and the variable flow time is evaluated pointwise through $a(b(x))$. We can interpret this as a time-reparametrization of Busemann gradient flows.

\section{Nonexpansiveness of Busemann gradient-descent layers}
\label{sec:Busemann layers nonexpansive}
In this section, we present the main theoretical result of our work, \Cref{trm:Ttau nonexpansive}, which provides a condition on the stepsize ensuring that the Busemann gradient-descent layer in \Cref{eq:Busemann gradient layer} is 1-Lipschitz. We present two versions of the proof. The first one relies only on basic properties of Busemann functions on Hadamard manifolds, and provides a strategy to prove the product-manifold extension presented in \Cref{trm:product-general}. The second one uses the construction in, e.g., \cite[Theorems~2.2,3.1]{arnold24bso}, \cite[Theorems~6,9]{ghirardelli2025conditional} based on variations through geodesics and Jacobi fields, providing a more geometric interpretation. 

\begin{theorem}
    \label{trm:Ttau nonexpansive}
    Let $(\M,g)$ be a Hadamard manifold and $b$ a Busemann function on $\M$. Let $\lambda \ge 0$, $\tau\geq 0$, $\beta\in\mathbb{R}$, and $\varphi : \mathbb{R}\to \mathbb{R} $ be a $\mathcal{C}^{1}$ function such that $\varphi'$ is globally Lipschitz continuous, i.e. $\varphi\in \mathcal{C}^{1,1}$. Assume
    \begin{equation*}
        \varphi'(s)\geq 0 \qquad\text{for all }s\in \mathbb R,    
    \end{equation*}
    and
    \begin{equation}
        \label{eq: condition for nonexpansiveness}
        0\leq \tau \lambda^2 \varphi''(s)\leq 2 \qquad\text{for a.e. }s\in \mathbb R,
    \end{equation}
    where $\varphi''$ denotes the a.e. derivative of the Lipschitz function $\varphi'$. Then, the map $T_\tau$ in \eqref{eq:Busemann gradient layer} is nonexpansive on $\M$.
\end{theorem}

Before proving the theorem, we provide an illustrative example of the bound and its tightness in \Cref{fig:tight-nonexpansive-bound}. More details on the calculations are in Appendix \ref{app:calculations-tight-bound}.
\begin{figure}[h!]
    \centering
    \includegraphics[width=0.8\linewidth]{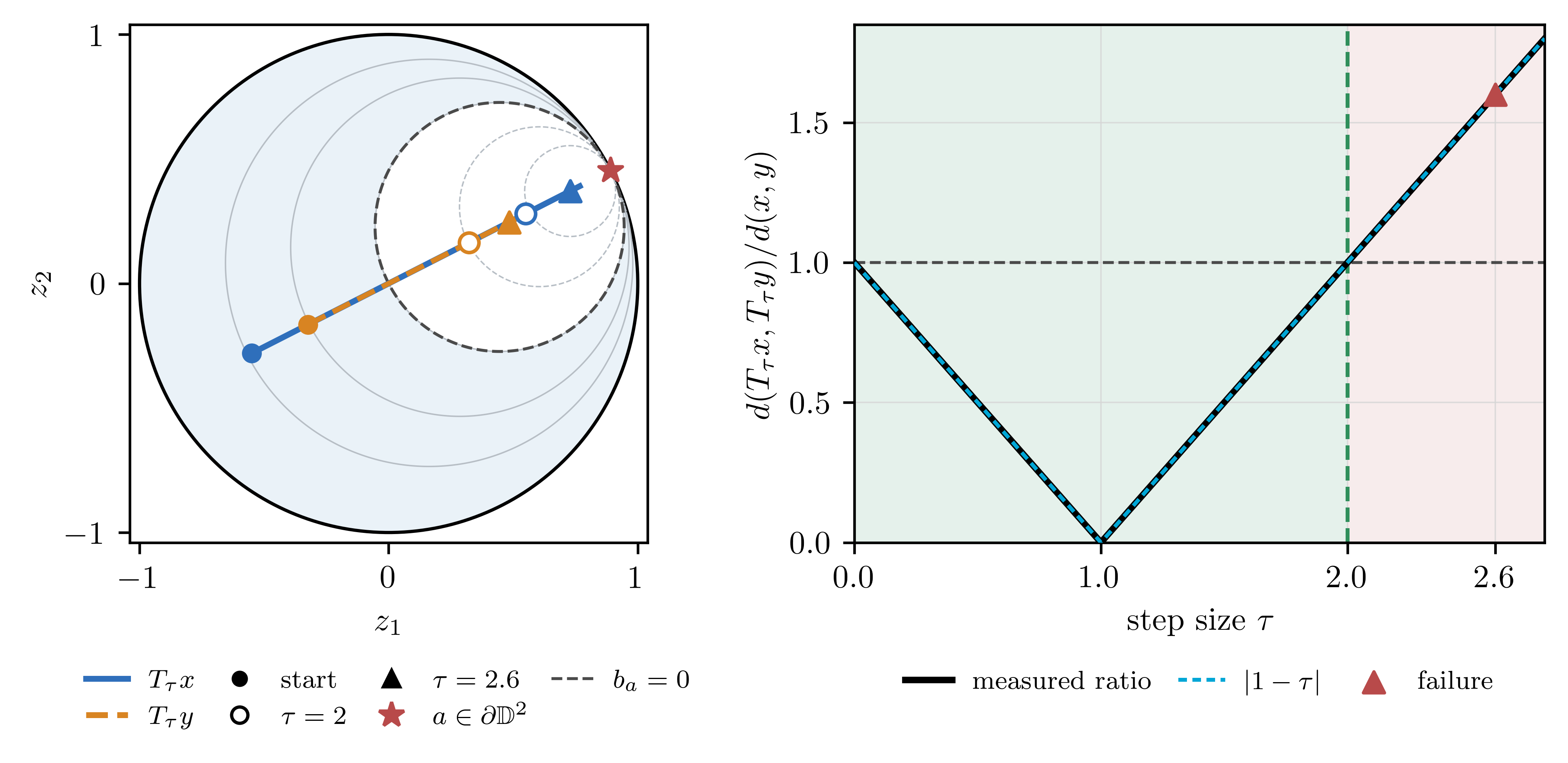}
    \caption{Sharpness example for a single Busemann layer on $\mathbb D^2$. For $V(x)=\frac12\operatorname{ReLU}(b_P(x))^2$, choose two active points $P_i=-r_iP$ on the diameter opposite $P$.  The update $T_\tau(x)=\Phi_{\tau\operatorname{ReLU}(b_P(x))}(x)$ reparametrizes the Busemann flow by the value of $b_P$, so the point with larger $b_P$ moves farther and the two trajectories cross.  On this pair, $
d(T_\tau(P_1),T_\tau(P_2))/d(P_1,P_2)=|1-\tau|$, showing nonexpansiveness for $0\le\tau\le2$ and failure beyond the bound.}
    \label{fig:tight-nonexpansive-bound}
\end{figure}

\begin{proof}[Proof of \Cref{trm:Ttau nonexpansive} (version 1)]
    If $\tau=0$, or $\lambda=0$, then $T_\tau=\mathrm{Id}$, and there is nothing to prove. Hence, we assume $\tau \lambda>0$ and that $\varphi$ is not constant. Define
    \begin{equation*}
        a(r) := \tau \lambda \varphi'(\lambda r+\beta).
    \end{equation*}
    Since $\varphi'$ is Lipschitz, the scalar map $a$ is Lipschitz on $\mathbb R$. Moreover, 
    $a'(r) = \tau \lambda^2 \varphi''(\lambda r+\beta)$ for a.e. $r\in\mathbb{R}$. 
    Hence
    \begin{equation*}
        0\leq a'(r)\leq 2 \qquad\text{for a.e. }r\in \mathbb R
    \end{equation*}
    whenever $\tau$ satisfies \eqref{eq: condition for nonexpansiveness}. Let $\gamma:[0,1]\to \M$ be the (unique) length-minimising geodesic segment from $x$ to $y$, i.e., $\gamma(0)=x$, $\gamma(1)=y$, and $\ell(\gamma)=d(x,y)$. Then, since
    \begin{equation*}
        d(T_\tau(x), T_\tau(y)) \leq \ell (T_\tau\circ \gamma),
    \end{equation*}
    it suffices to show that $\ell (T_\tau\circ \gamma)\leq \ell (\gamma)$. Define $w(s) := b(\gamma(s))$. Since $(t,x)\mapsto\Phi_t(x)$ is $\mathcal C^1$, the curve $s\mapsto \Phi_{a(w(s))}(\gamma(s))=T_\tau(\gamma(s))$ is absolutely continuous. Therefore, it is enough to show that $\|(T_\tau\circ \gamma)'(s)\| \leq \|\dot{\gamma}(s)\|$ for a.e. $s\in [0,1]$ so that
    \begin{equation*}
        \ell (T_\tau\circ \gamma) = \int_0^1 \left\| (T_\tau\circ \gamma)'(s) \right\| \dd s \leq \int_0^1 \|\dot{\gamma}(s)\| \dd s = \ell (\gamma).
    \end{equation*}
    It follows from \Cref{lemma: phi coincides with exp} that $T_\tau(\gamma(s)) = \Phi_{a(b( \gamma(s)))}(\gamma(s))$, where $(\Phi_t)_{t\geq 0}$ is the Busemann flow in \eqref{eq: Busemann flow phi}. Since $w$ is $\mathcal C^1$ and $a$ is Lipschitz, the
    curve $a\circ w$ is absolutely continuous. By the one-dimensional chain rule for Lipschitz functions \cite[Corollary~3.68]{leoni2017first}, $(a\circ w)'(s)=a'(w(s))w'(s)$ for a.e. $s\in[0,1]$. Therefore, for almost every $s\in [0,1]$, one has
    \begin{equation*}
        (T_\tau \circ \gamma)'(s) = D\Phi_{a(w(s))}(\dot{\gamma}(s)) + a'(w(s))\,w'(s)\, \frac{\partial}{\partial t}\Phi_t(\gamma(s))\bigg|_{t=a(w(s))}.
    \end{equation*}
     We have $w'(s) = \langle \grad b(\gamma(s)), \dot{\gamma}(s) \rangle$, i.e. $w'(s)$ is the component of $\dot{\gamma}(s)$ in the direction of $\grad b(\gamma(s))$. For every $s\in [0,1]$, let $\xi_H(s) \in \mathrm{span}\{\grad b(\gamma(s))\}^\perp$ be the orthogonal component so that
    \begin{equation*}
        \dot{\gamma}(s) = w'(s) \cdot \grad b(\gamma(s)) + \xi_H(s).
    \end{equation*}
    Then,
    \begin{equation*}
        \|\dot{\gamma}(s)\|^2 = |w'(s)|^2 + \|\xi_H(s)\|^2.
    \end{equation*}
    By the properties of pushforward maps, $D\Phi_t(\grad b(\cdot)) = \grad b(\Phi_t(\cdot))$ and hence 
    \begin{equation*}
        D\Phi_{a(b(\gamma(s)))}(\grad b(\gamma(s))) = \grad b(\Phi_{a(b(\gamma(s)))}(\gamma(s))) = \grad b(T_\tau(\gamma(s))).
    \end{equation*}
    Similarly,
    \begin{equation*}
        \frac{\partial}{\partial t}\Phi_t(\gamma(s))\bigg|_{t=a(b(\gamma(s)))}  = -\grad b(\Phi_{a(b(\gamma(s)))}(\gamma(s))) = -\grad b(T_\tau(\gamma(s))).
    \end{equation*}
    Putting everything together, we have
    \begin{align*}
        (T_\tau \circ \gamma)'(s) &= D\Phi_{a(b(\gamma(s)))}(\dot{\gamma}(s)) - a'(w(s)) w'(s) D\Phi_{a(b(\gamma(s)))}(\grad b(\gamma(s)))\\
        &= D\Phi_{a(b(\gamma(s)))}\Big(\dot{\gamma}(s)-a'(w(s)) w'(s)\grad b(\gamma(s))\Big).
    \end{align*}
    Since $\dot\Phi_t=X\circ\Phi_t$ with $X=-\grad b$, and since
    $b\in \mathcal C^2$ is convex, we have
    \[
        \langle \nabla_\eta X,\eta\rangle
        =
        -\langle \nabla_\eta \grad b,\eta\rangle
        =
        -\mathrm{Hess}\,b(\eta,\eta)
        \le 0 .
    \]
    Thus $X$ satisfies the Riemannian monotonicity condition of
    \cite{ghirardelli2025conditional} with constant $\nu=0$. Hence its exact
    flow is nonexpansive. Since $X$ is $\mathcal C^1$, the flow is differentiable
    with respect to the initial condition, and the corresponding differential
    satisfies
    \[
        \|D\Phi_t(x)\eta\|\le \|\eta\|,
        \qquad t\ge 0,\ \eta\in T_x\M .
    \]
    Applying this estimate with $x=\gamma(s)$ and $t=a(b(\gamma(s)))$ gives
    \begin{equation}\label{eq:nonexpExact}
        \bigl\|D\Phi_{a(b(\gamma(s)))}(\gamma(s))\eta\bigr\|
        \le
        \|\eta\| ,
    \end{equation}
    provided $a(b(\gamma(s)))\ge 0$. This is why we assume $a\ge 0$,
    equivalently $\varphi'\ge 0$. Therefore,
    \begin{equation*}
        \|(T_\tau\circ \gamma)'(s)\|    \leq
        \bigl\| \dot{\gamma}(s) - a'(w(s)) w'(s) \grad b(\gamma(s)) \bigr\|.
    \end{equation*}
    Using the orthogonal decomposition of $\dot{\gamma}(s)$ at the source point,
    \begin{equation*}
        \dot{\gamma}(s) - a'(w(s)) w'(s) \grad b(\gamma(s)) =
        \xi_H(s) + (1-a'(w(s))) w'(s) \grad b(\gamma(s)).
    \end{equation*}
    Since $\xi_H(s)\perp \grad b(\gamma(s))$, we obtain
    \begin{equation*}
        \bigl\|\dot{\gamma}(s) - a'(w(s)) w'(s) \grad b(\gamma(s)) \bigr\|^2 =
        \|\xi_H(s)\|^2 + |w'(s)|^2 (1-a'(w(s)))^2.
    \end{equation*}
    Finally, since $0\leq a'(r)\leq 2$ for a.e. $r\in \mathbb R$, it follows that
    \begin{equation*}
        \bigl\| \dot{\gamma}(s) - a'(w(s)) w'(s) \grad b(\gamma(s)) \bigr\|^2 \leq
        \|\xi_H(s)\|^2 + |w'(s)|^2 = \|\dot{\gamma}(s)\|^2,
    \end{equation*}
    which implies
    \begin{equation*}
        \|(T_\tau \circ \gamma)'(s)\| \leq \|\dot{\gamma}(s)\|.
    \end{equation*}
    This concludes the proof. 
\end{proof}

\begin{remark}
    The first version of the proof applies under the stated lower-regularity assumptions $\varphi\in\mathcal{C}^{1,1}$.
    The second proof is included as a more geometric argument, and is written
    under the additional simplifying assumption $\varphi\in \mathcal C^2$, so that
    $a(r)=\tau\lambda\varphi'(\lambda r+\beta)$ is $\mathcal C^1$ and the computations
    involving $a'$ are classical.
\end{remark}

\begin{proof}[Proof of \Cref{trm:Ttau nonexpansive} (version 2)] 
    Define
    \begin{equation*}
        a(r) := \tau \lambda \varphi'(\lambda r+\beta).
    \end{equation*}
    In this proof, we assume $\varphi\in \mathcal C^2$. Then
    \[
        a'(r)=\tau\lambda^2\varphi''(\lambda r+\beta)
    \]
    for every $r\in\mathbb R$. Moreover, $a(r)\geq 0$ and since $\varphi'$ is Lipschitz, the scalar map $a$ is Lipschitz on $\mathbb R$. Fix $x,y\in \M$ and let $\sigma:[0,1]\to \M$ be the minimizing geodesic from $x$ to $y$, so that 
    \begin{equation*}
        \sigma(0)=x,\quad \sigma(1)=y,\qquad \|\dot\sigma(s)\|\equiv d(x,y).
    \end{equation*}
    Define a two-parameter map
    \begin{equation*}
        \Gamma:[0,1]\times[0,1]\to \M,\qquad
        \Gamma(s,t) := \Phi_{t a(b(\sigma(s)))}\big(\sigma(s)\big).
    \end{equation*}
    For each fixed $s$, the curve $t\mapsto \Gamma(s,t)$ is the Busemann geodesic through $\sigma(s)$, starting at $\Gamma(s,0)=\sigma(s)$ and ending at
    \begin{equation*}
        \Gamma(s,1) = \Phi_{a(b(\sigma(s)))}(\sigma(s)) = \exp_{\sigma(s)} (-a(b(\sigma(s))) \, \grad b(\sigma(s)) )=  T_\tau(\sigma(s)).
    \end{equation*}
    On the other hand, for each fixed $t$, the curve $s\mapsto \Gamma(s,t)$ is a deformation of $\sigma$ into the image curve
    \begin{equation*}
        \eta(s) := \Gamma(s,1) = T_\tau(\sigma(s)).
    \end{equation*}
    For each $s\in[0,1]$, the variational field along $t\mapsto \Gamma(s,t)$ defined by
    \begin{equation*}
        J^s(t) := \partial_s \Gamma(s,t)
    \end{equation*}
    is a Jacobi field. In particular,
    \begin{equation*}
        J^s(0) = \partial_s \Gamma(s,0) = \dot\sigma(s),\qquad
        J^s(1) = \partial_s \Gamma(s,1) = \dot\eta(s).
    \end{equation*}
    Our goal is to show that
    \begin{equation}
    \label{eq:lengths}
        \int_0^1 \|J^s(1)\|\,\dd s \le \int_0^1 \|J^s(0)\|\,\dd s,
    \end{equation}
    which implies $d(T_\tau (x),T_\tau (y))\le d(x,y)$. Fix $s=0$ and consider the $t$-curve
    \begin{align*}
        \gamma(t) := \Gamma(0,t), \qquad \gamma(0) &= x,  \\
        \gamma(1) & = \exp_x (-a(b(x)) \grad b(x)) = \exp_x (-\tau \grad V(x)) .
    \end{align*}
    By construction, $\gamma$ is a (reparametrized) Busemann geodesic
    \begin{align*}
        \gamma(t) &= \Phi_{t a(b(x))}(x),\\
        \dot\gamma(t) &= a(b(x)) \left. \frac{\dd}{\dd s} \Phi_{s}(x) \right|_{s=ta(b(x))}
        = - a(b(x))\,\grad b(\gamma(t)),
    \end{align*}
    with $\|\Dot{\gamma}\| = |a(b(x))|$ independent of $t$, and $\nabla_{\Dot{\gamma}} \Dot{\gamma}=0$, hence $\dot\gamma$ is parallel along $\gamma$.
    Set
    \[
        u(t):=\grad b(\gamma(t)).
    \]
    Since $\nabla_{\grad b}\grad b=0$, the vector field $u$ is parallel
    along $\gamma$. We decompose
    \[
        J(t) = J^\parallel(t)+J^\perp(t) = (c_0+c_1t)u(t)+J^\perp(t),
        \qquad \langle J^\perp(t),u(t)\rangle=0.
    \]
    The coefficients $c_0$ and $c_1$ depend on the initial data at $t=0$. The torsion-free commutation identity gives
    \begin{equation*}
        D_t \partial_s \Gamma (s,t) = D_s \partial_t \Gamma (s,t),
    \end{equation*}
    hence
    \begin{align*}
        D_t J(0)
        &= \left. D_t \partial_s \Gamma (s,t) \right|_{t=s=0} 
        = \left. D_s (-a(b(\sigma(s)))\,  \grad b(\sigma(s))) \right|_{s=0} \\
        & = - 
        \left. \left( \frac{\dd}{\dd s} a(b(\sigma(s))) \right) \, \grad b(\sigma(s))\right|_{s=0}
        - \left. a(b(\sigma(s))) \, D_s ( \grad b(\sigma(s))) \right|_{s=0} \\
        &= - a'(b(x)) \, \langle \grad b(x), S_0 \rangle \, \grad b(x) - a(b(x)) \nabla_{S_0} \grad b(x).
    \end{align*}
    Then $c_0$ and $c_1$ are given by \cite[Section~3]{ghirardelli2025conditional}
    \begin{align*}
        c_0 &= \langle J(0), u(0) \rangle = \langle S_0, \grad b(x) \rangle \\
        c_1 &= \langle D_t J(0), u(0) \rangle 
        = - a'(b(x)) \, \langle S_0, \grad b(x) \rangle
        = - a'(b(x)) \, c_0,
    \end{align*}
    where we have used
    \begin{align*}
        \langle \nabla_{S_0} \grad b, \grad b \rangle 
        &= \Hess b (S_0, \grad b) \\
        &= \Hess b (\grad b, S_0)
        =  \langle \nabla_{\grad b}\grad b, S_0\rangle = 0,
    \end{align*}
    since $ \nabla_{\grad b}\grad b=0$. Therefore
    \begin{align*}
        \|J^\parallel(1)\|^2 - \|J^\parallel(0)\|^2 = (2c_0 c_1 + c_1^2) = c_0^2 \, a'(b(x))\, (-2 + a'(b(x)) ) \le 0,
    \end{align*}
    as $0\le a'(r) \le 2$ for any $r\in\mathbb{R}$ by assumption. 
    Since $u$ is parallel along $\gamma$,
    \[
        D_t J^\perp = D_t J(t)-c_1u(t).
    \]
    Using $D_t J = D_s\partial_t\Gamma\big|_{s=0}$ and $\partial_t\Gamma(s,t) = -a(b(\sigma(s)))\grad b(\Gamma(s,t))$ we obtain
    \begin{align*}
        D_t J(t)
        &= -\left. \frac{\dd}{\dd s}a(b(\sigma(s))) \right|_{s=0}u(t)
        -a(b(x))\nabla_{J(t)} \grad b (\gamma(t) ) \\
        &= c_1 u(t)-a(b(x))\nabla_{J(t)}\grad b (\gamma(t)).
    \end{align*}
    Consequently,
    \[
        D_t J^\perp(t) = -a(b(x))\nabla_{J(t)}\grad b (\gamma(t)).
    \]
    Since
    \[
        J(t)=(c_0+c_1t)\grad b(\gamma(t)) +J^\perp(t)
    \]
    and $\nabla_{\grad b}\grad b=0$, we have
    \[
        \nabla_{J(t)} \grad b(\gamma(t)) = \nabla_{J^\perp(t)}\grad b(\gamma(t)).
    \]
    Hence
    \[
        D_t J^\perp (\gamma(t))
        = -a(b(x))\nabla_{J^\perp(t)}\grad b  (\gamma(t)).
    \]
    Together with the convexity of $b$ by \Cref{lemma: properties of Busemann functions on Hadamard manifolds}, we then obtain
    \begin{align*}
        \frac{d}{dt}\|J^\perp (t)\|^2 & 
        = 2 \, \langle \nabla_{\dot \gamma(t)} J^\perp(t), J^\perp(t) \rangle \\
        & 
        = -2 a(b(x)) \, \Hess b_{\gamma(t)} (J^\perp(t), J^\perp(t)) \le 0,
    \end{align*}
    since $a(r) \ge 0$ for any $r\in\mathbb{R}$ by assumption. In particular, 
    \begin{equation*}
        \|J^\perp(1) \|^2 - \|J^\perp(0) \|^2  \le 0.
    \end{equation*} 
    Combining the two estimates and using the orthogonality of the decomposition, we obtain
    \begin{align*}
        \|J(1)\|^2-\|J(0)\|^2 =
        \left(
            \|J^\parallel(1)\|^2-\|J^\parallel(0)\|^2
        \right) +
        \left(
            \|J^\perp(1)\|^2-\|J^\perp(0)\|^2
        \right)
        \le0.
    \end{align*}
    Repeating the same reasoning to all $s\in[0,1]$ gives $\| J^s(1)\|^2 - \| J^s(0)\|^2 \le 0 $ (equivalently $\| J^s(1)\| - \| J^s(0)\| \le 0 $ ) for any $s\in[0,1]$. Hence 
    \begin{align*}
        d \big(T_\tau(x),T_\tau(y)\big) \le 
        \int_0^1 \|J^s(1)\|\, \dd s \le
        \int_0^1 \|J^s(0)\|\, \dd s = L(\sigma) = d(x,y),
    \end{align*}
    and $T_\tau$ is nonexpansive.
\end{proof}

\begin{remark}[Combination or composition of potentials]
    \label{remark:Lie-Trotter}
    Let $\omega_i \ge 0$,  $\lambda_i \ge 0$, and $\beta_i \in \R$. 
    Let $b_i$ be Busemann functions on $\M$ and let $\varphi : \R \to \R$ of class $\C^{1,1}$ and such that
    \[
        \esssup_{s\in\mathbb R} \varphi''(s)\leq M_2.  
    \]
    Consider the maps
    \[
        T_i(x)=\exp_x(-\tau_i\grad V_i(x)), \qquad
        V_i(x) = w_i\varphi(\lambda_i b_i(x)+\beta_i), \qquad i=1,\dots,m.
    \]
    If $\tau_i w_i\lambda_i^2M_2\leq 2$ for every $i$ then every $T_i$ is nonexpansive, and the composition $T=T_m\circ T_{m-1}\circ\cdots\circ T_1$ is nonexpansive too. Consider now the simultaneous step
    \[
        x\mapsto \exp_x\left(-\tau\grad\sum_{i=1}^mV_i(x)\right).
    \]
    Here, several Busemann directions are mixed inside one exponential, destroying the one-dimensional horospherical structure used above. The split layer keeps each factor in the single-Busemann regime, where the gradient step is a variable-time gradient flow and can be analyzed more easily. However, with $\tau_1=\cdots=\tau_m=\tau$ small, the split version provides an approximation of the simultaneous step (Lie-Trotter splitting, \cite{blanes2024splitting}), making the split version as expressive as the more complex simultaneous-step one, while still allowing nonexpansiveness certificates.
\end{remark}

The map $T_\tau$ is a nonexpansive map on the finite-dimensional Hadamard manifold $\M$. Moreover, \Cref{trm: Ttau is quasi alpha FNE} with $\bar \tau$ given by \eqref{eq: condition for nonexpansiveness} guarantees $T_\tau$ is quasi-$\alpha$-firmly nonexpansive on $\mathcal{M}$, and \Cref{corollary: convergence to fix point for quasi firmly nonexpansive operators} implies that the iterates $x_{k+1}=T_\tau(x_k)$, $x_0\in\M$, converge strongly to a fixed point of $T_\tau$ in $\mathcal{M}$, provided $\Fix T_\tau \ne \emptyset$. Furthermore, this property is preserved under composition, provided we assume non-trivial intersections of the fixed sets; see \Cref{corollary: finite composition of firmly non expansive is firmly nonexpansive}. This allows us to recover quasi-$\alpha$-firmly nonexpansive networks in two different ways: by composing layers as $T_\tau$ with suitably constrained stepsizes and ensuring nontrivial intersections of their fixed sets, or defining a modified network $N_\theta(x)=x\#_\alpha D_\theta(x)$ for a nonexpansive network $D_\theta(x)=T_{\tau_m}\circ \cdots \circ T_{\tau_1}$ and $\alpha\in(0,1)$. In the inverse problem considered in \Cref{subsec:spd-inverse}, we opt for the second option since it is not immediate to ensure the desired condition on the fixed sets of the layers.

\subsection{Generalization to product manifolds}
The construction seen so far can be extended to product manifolds for which the base manifold admits efficiently computable Busemann functions and their gradients. This extension can be used to apply the proposed theory to hyperbolic-valued images and graphs, or for considering tasks on product manifolds, such as in Diffusion Tensor Imaging (DTI). Such applications are outside the scope of this paper and will be considered in future extensions.

Let $\N=\M^m$ be endowed with the product Riemannian metric. We adopt the notation $d_\M$ and $d_\N$ for the Riemannian distances induced on $\M$ and $\N$, respectively. Let $X=(x^1,\dots,x^m),Y=(y^1,\dots,y^m)\in\N$ denote a generic pair of points on the product manifold. The product distance is
\[
    d_\N(X,Y)^2=\sum_{k=1}^m d_\M(x^k,y^k)^2.
\]
For a tangent vector $(v_1,\dots,v_m)\in T_{x_1} \M\times\cdots\times T_{x_m} \M \simeq T_X \N$, the product exponential is componentwise:
\[
    \exp_X(v^1,\dots,v^m)
    =\bigl(\exp_{x^1}(v^1),\dots,\exp_{x^m}(v^m)\bigr).
\]
Fix one Busemann function $b:\M\to\mathbb{R}$ and define the map
\[
    B:\N\to\mathbb{R}^m,
    \qquad
    B(X)=\bigl(b(x^1),\dots,b(x^m)\bigr).
\]
Let $U:\mathbb{R}^m\to\mathbb{R}$ and define $V(X)=U(B(X))$. Assuming $U\in \mathcal C^1$, the product gradient has components
\[
    \grad_{x^k}V(X)
    =\partial_k U(B(X))\grad b(x^k),
    \qquad k=1,\dots,m.
\]
Hence the explicit Riemannian gradient step of $V$ is
\[
    T(X)=\exp_X(-\tau\grad V(X)),
\]
with components
\begin{equation}
    \label{eq:product-single-busemann-step}
    T(X)^k
    =\exp_{x^k}\left(-\tau\partial_kU(B(X))\grad b(x^k)\right)
    =\Phi_{\tau\partial_kU(B(X))}(x^k),
\end{equation}
where we have used \Cref{lemma: phi coincides with exp}. Note that the flow times for each component depend on the full vector $B(X)$, so the components of $T$ are coupled through $U\circ B$. We study the nonexpansiveness of this gradient step with the two results below.
\begin{theorem}
    \label{trm:product-general}
    Let $(\M,g)$ be a Hadamard manifold and $\N = \M^m$ the product manifold equipped with the product metric. Let $a:\R^m\to\R^m$ be Lipschitz continuous as a vector-valued function. Assume
    \[
        a_k(z)\geq0, \qquad k=1,\dots,m, \qquad z\in\R^m,
    \]
    and assume that $R(z):=z-a(z)$ is nonexpansive in the Euclidean norm on $\R^m$.
    Define
    \[
         T:\N\to\N,
         \qquad
         T(X)^k=\Phi_{a_k(B(X))}(x^k),
         \qquad k=1,\dots,m.
    \]
    Then $T$ is nonexpansive on $\N$.
\end{theorem}
The above \Cref{trm:product-general} can be proved by adapting version 1 of the proof of \Cref{trm:Ttau nonexpansive}. We provide the complete proof in \Cref{appendix: proofs omitted in the main document}.

\begin{remark}
    For $m=1$, nonexpansiveness of $r\mapsto r-a(r)$ is equivalent to $|1-a'(r)|\leq1$ a.e., hence to the one-dimensional criterion $0\leq a'(r)\leq2$ a.e. in \Cref{trm:Ttau nonexpansive}. Therefore, \Cref{trm:product-general} provides a strict generalization of \Cref{trm:Ttau nonexpansive}.
\end{remark}

\begin{corollary}
    \label{cor:gradient-step-product}
    Let $(\M,g)$ be a Hadamard manifold and $\N = \M^m$ the product manifold equipped with the product metric. Let $U:\R^m\to\R$ be $\C^1$, convex, $L$-smooth, and nondecreasing in every component, i.e., $\nabla U(z)\geq 0$ componentwise. Define $V(X)=U(B(X))$. Then the explicit Riemannian gradient step
    \[
        T(X)=\exp_X(-\tau\grad V(X))
    \]
    is nonexpansive on $\N$ for every $\tau\in[0,2/L]$. 
\end{corollary}
We recall that a $\C^1$ function $U:\R^m\to\R$ is $L-$smooth if its gradient $\nabla U:\R^m\to\R^m$ is $L-$Lipschitz.
\begin{proof}
    For convex $L$-smooth $U$, the Euclidean map $z\mapsto z-\tau\nabla U(z)$ is nonexpansive for $\tau\in[0,2/L]$. We conclude by applying \Cref{trm:product-general} with $a(z)=\tau\nabla U(z)$.
\end{proof}

\section{Implementation of the Busemann gradient-descent layers}
\label{sec:implementation}
In this section, we provide some practical implementations of the proposed Busemann gradient-descent layers. We provide explicit expressions of Busemann functions in the three settings $\M=\mathbb{R}^n$, the $n-$dimensional Poincar\'e ball $\M=\mathbb{D}^n$, and the manifold $S_{++}(n)$ of symmetric positive definite (SPD) matrices. We then list some admissible activation functions (i.e., those that satisfy the hypothesis of \Cref{trm:Ttau nonexpansive}) and compute the corresponding bound on the stepsize. The proposed architectures on the Poincar\'e ball and SPD manifold are then empirically tested in \Cref{sec:numerical experiments}.

\subsection{Examples of Busemann functions}
\subsubsection{Euclidean space}
For every $u\in \mathbb{R}^n$ with $\|u\|_2=1$, the Busemann function associated with the geodesic ray $\gamma(t) = -tu$ is $b(x)=u^\top x$ (see \cite[Example II.8.24(1)]{bridson2013metric}).
The norm constraint can be removed by absorbing the norm in $\lambda$: if $\xi \in\mathbb{R}^n\setminus\{0\}$, then $\xi^\top x = \|\xi\|_2 b_{\xi/\|\xi\|_2}(x)=:\lambda b_{\xi/\|\xi\|_2}(x)$. 
Let $\beta \in \mathbb{R}$ and $\varphi:\mathbb{R} \to \mathbb{R}$ in $\C^{1,1}$ be a convex function. We consider the potential
\begin{equation*}
    V(x) = \varphi(\xi^\top x + \beta).
\end{equation*}
Set $\Delta := \esssup_{s\in \R}\varphi''(s)$, which is finite given that $\varphi'$ is Lipschitz continuous. The function $V$ defined above is $L$-smooth with $L=\|\xi\|_2^2\Delta$. The gradient-descent map associated to it is therefore $1$-Lipschitz if $\tau \in [0,2/L]$ and it writes
\begin{equation}
    \label{eq:resnetEuclidean}
    T_\tau(x) = x - \tau \nabla V(x) = x- \tau \, \varphi'(\xi^\top x + \beta) \xi.
\end{equation}
In this context, $\theta=(\xi, \beta)$ are the trainable weights of the layer. Layers defined as \eqref{eq:resnetEuclidean} can recover Householder reflection layers \cite[Eq. 6]{mhammedi2017efficient} and the \texttt{MaxMin} and \texttt{GroupSort} activations \cite{murari2026approximation,anil2019sorting}. We remark that, in the case of $\R^n$, the potential obtained by summing finitely many such potentials also leads to nonexpansive gradient steps:
\begin{align*}
V(x) &= \sum_{i=1}^k \varphi(\xi_i^\top x + \beta_i),\,\,T_\tau(x) = x-\tau \Xi^\top \varphi'(\Xi x + v),\\
\Xi &= \begin{bmatrix} \xi_1^\top \\ \vdots \\ \xi_k^\top \end{bmatrix}, \quad v=\begin{bmatrix} \beta_1 \\ \vdots \\ \beta_k\end{bmatrix},\quad \tau \in \left[0,\frac{2}{\Delta\|\Xi\|_2^2}\right].
\end{align*}
This is a layer of a $1$-Lipschitz ResNet as considered in \cite{meunier2022dynamical,sherry24dsn}. 
\begin{remark}
For $\M=\mathbb{R}^n$, the considered Busemann functions are linear. Therefore, here, the non-decreasing assumption on $\varphi$ can be dropped. It suffices to assume that $\varphi\in \mathcal{C}^{1,1}$ and convex.
\end{remark}

\subsubsection{Hyperbolic manifolds}\label{sec:hypArch} 
In the Poincar\'e ball $\mathbb{D}^n = \{ x\in\R^n : \|x\|_2 < 1\}$, the points at infinity are the ones lying on the hypersphere $S^{n-1} = \partial \mathbb{D}^n$. Let $\xi \in S^{n-1}$ and consider a geodesic ray $\gamma$, parametrized by arc length, and tending to $\xi$. The Busemann function associated with $\gamma$ is
\begin{equation}
    \label{eq: Busemann function in Poincare ball}
    b(x) = \lim_{t\to\infty} (d(\gamma(t), x)-t) =  \log \left(\frac{\|\xi- x\|_2^2}{1-\|x\|_2^2} \right),
\end{equation}
where $\|\xi\|_2=1$ and $x\in\mathbb{D}^n$ (see \cite{ghadimi2021hyperbolic}).
The Euclidean gradient of $b$ is
\begin{align*}
    \nabla b (x) 
    &= \frac{2}{1-\|x\|_2^2}x + \frac{2}{\|x-\xi\|_2^2} (x-\xi).
\end{align*}
The Riemannian metric on $\mathbb{D}^n$ is defined as
\begin{equation*}
    g_x^\mathbb{D} = \left( \frac{2}{1-\|x\|_2^2} \right)^2 g^E =: \delta_x^2 g^E
\end{equation*}
where $g^E$ is the Euclidean one. The Riemannian gradient $\grad b$ is then 
\begin{equation*}
    \grad b (x) = \frac{1}{\delta_x^2} \nabla b(x) = \frac{1-\|x\|_2^2}{2} x + \frac{(1-\|x\|^2_2)^2}{2\|x-\xi\|_2^2} (x-\xi).
\end{equation*}
Using $b$ and $\grad b$ as above yields an implementable gradient-descent map, with the freedom to choose $\xi\in S^{n-1}$, $\lambda>0$, and $\beta\in \R$. We provide a more implementation-oriented description of this layer in \Cref{alg:hyp-single-busemann}.

\subsubsection{SPD matrices}\label{se:spd_def}
Consider the manifold of symmetric positive definite matrices
\[
    S_{++}(n)
    =
    \{X\in\mathbb{R}^{n\times n}: X^\top=X,\ X\succ0\}
\]
equipped with the affine-invariant Riemannian distance
\cite{pennec2006riemannian}
\[
    d_{\mathrm{AI}}(X,Y)
    =
    \left\|
    \Log\left(X^{-1/2}YX^{-1/2}\right)
    \right\|_F
    =
    \left(\sum_{i=1}^n \log^2 \lambda_i\right)^{1/2},
\]
where $\lambda_1,\dots,\lambda_n$ are the eigenvalues of
\(X^{-1/2}YX^{-1/2}\). Then $S_{++}(n)$ is a Hadamard manifold, and
$T_XS_{++}(n)\simeq S(n)$, the space of $n\times n$ symmetric matrices.

For the split SPD layers, we use the explicit Cholesky-based expression of
Busemann functions on $S_{++}(n)$. This is the formula appearing, for
example, in \cite[Example~4]{ferreira2026subdifferential}. Let $U\in O(n)$, let
$d\in\mathbb{R}^n$ be a vector with $\|d\|_2=1$ and whose entries are sorted in ascending order, fixing the
ordering convention used in the Cholesky factorization. We write
\[
    U^\top XU = LL^\top,
\]
where $L$ is lower triangular factor in the Cholesky factorization of $U^\top X U$, which has positive diagonal. We define the Busemann function
\[
    b_{U,d}(X)
    =
    -2\sum_{i=1}^n d_i\log L_{ii}.
\]
A geodesic ray starting from the identity is determined by a unit tangent
direction \(A\in S(n)\):
\[
    \gamma_A(t)=\exp_I(tA)=\exp(tA), \qquad \|A\|_F=1,
\]
where $\exp$ without a subscript denotes the matrix exponential. We parametrize such directions as
\[
    A = U\operatorname{diag}(d)U^\top,
    \qquad U\in O(n), \quad \|d\|_2=1,\,\,d_1<d_2<\dots<d_n.
\]
In the implementation, \(U\) is learned through an orthogonal parametrization.
The vector \(d\) is obtained by sorting a raw trainable vector, subtracting its
mean, and normalizing it. Thus \(\sum_i d_i=0\) and \(\|d\|_2=1\). This
restricts the associated direction $A$ to be trace-free. This centering is a modeling choice for the
SPD denoiser, and it is not required by the general Busemann formula on full
$S_{++}(n)$.

The affine-invariant Riemannian gradient of \(b_{U,d}\) is
\[
    \operatorname{grad} b_{U,d}(X)
    =
    - ULDL^\top U^\top,
    \qquad
    D=\operatorname{diag}(d),
\]
see \cite[Equation 23]{ferreira2026subdifferential}.
For a scalar potential
\[
    V(X)
    =
    \varphi(\lambda b_{U,d}(X)+\beta),
    \qquad
    \lambda>0,
\]
a negative-gradient step has tangent direction
\[
    -\tau\,\operatorname{grad}V(X)
    =
    s\, ULDL^\top U^\top,
    \qquad
    s
    =
    \tau\lambda\,
    \varphi'(\lambda b_{U,d}(X)+\beta).
\]
In our experiments $\varphi(t)=\frac12\mathrm{ReLU}(t)^2$, so
$\varphi'(t)=\mathrm{ReLU}(t)$.

Using the affine-invariant exponential map and congruence invariance,
\[
    \exp_{LL^\top}(sLDL^\top)
    =
    L\exp(sD)L^\top.
\]
Equivalently, the Cholesky factor is updated as
\[
    L
    \longmapsto
    L\exp\left(\frac{s}{2}D\right),
\]
which gives the closed-form split update
\begin{equation}\label{eq:gradStepSPD}
    \exp_X(-\tau \grad V(X))
    =
    U
    \left(
        L\exp\left(\frac{s}{2}D\right)
    \right)
    \left(
        L\exp\left(\frac{s}{2}D\right)
    \right)^\top
    U^\top.
\end{equation}
This is the update implemented in the split Busemann SPD denoiser, which is more computationally efficient than directly computing the affine-invariant Riemannian exponential at $X$.

The trainable parameters of each single-Busemann step are the orthogonal matrix
$U$, the vector $d$, the scalar offset $\beta$, the positive
scale $\lambda$, and the bounded stepsize $\tau$. 

\subsection{Examples of activations}
Our theory is compatible with activation functions $\varphi:\R\to\R$ that are $\mathcal C^1$, have a globally Lipschitz derivative, are convex, and are non-decreasing. Some admissible ones are:
\begin{itemize}
    \item \textit{Softplus.}
    \begin{equation*}
        \varphi(r) = \log(1+e^r), \qquad
        \varphi'(r) = \frac{e^r}{1+e^r} \ge 0, \qquad
        \varphi''(r)\leq \frac14,
    \end{equation*}
    and condition \eqref{eq: condition for nonexpansiveness} becomes
    \begin{equation*}
        0\le \tau \lambda^2\leq 8 \iff 0 \le \tau \leq \frac{8}{\lambda^2}.
    \end{equation*}
    \item \textit{Squared $\relu$.}
    \begin{equation*}
        \varphi(r)=\frac12 {\relu(r)}^2, \qquad
        \varphi'(r)=\relu (r), \qquad
        \varphi''(r)= \begin{cases}
        0,& r<0,\\
        1,& r>0,
        \end{cases}
    \end{equation*}
    with $\varphi''$ defined for almost every $r$ since $\varphi'$ is Lipschitz continuous. Condition \eqref{eq: condition for nonexpansiveness} is then
    \begin{equation*}
        0 \le \tau \lambda^2\leq 2 \iff 0 \le \tau \leq \frac{2}{\lambda^2}.
    \end{equation*}
    \item \textit{Smooth bounded-slope activations.}
    Any convex nondecreasing $\C^{1,1}$ activation with
    \begin{equation*}
        \esssup_{r\in\mathbb R} \varphi''(r)\leq M_2
    \end{equation*}
    fits the theorem and yields the stepsize restriction
    \begin{equation*}
        0 \le \tau \lambda^2 M_2 \leq 2 \iff 0 \le \tau \leq \frac{2}{\lambda^2 M_2}.
    \end{equation*}
\end{itemize}
In our numerical experiments, we consider $\mathrm{ReLU}^2/2$, which has empirically performed slightly better than softplus. We plot the two activations side-by-side in \Cref{fig:activations} together with their first two derivatives.

\begin{figure}[h!]
    \centering
    \includegraphics[width=0.7\linewidth]{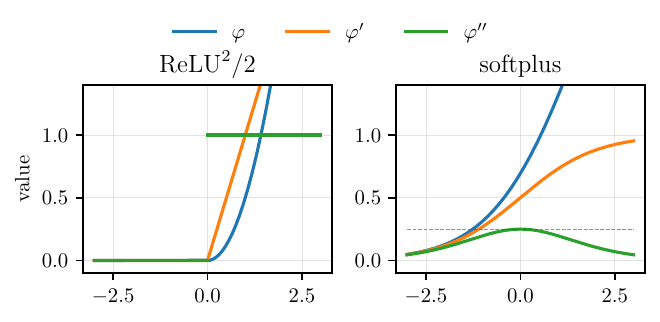}
    \caption{Two activation functions that satisfy the assumptions of our theory, along with their first two derivatives.}
    \label{fig:activations}
\end{figure}

\begin{remark}
The above are admissible activation functions under the $\C^{1,1}$-regularity assumption. If $\mathcal{C}^2$-regularity is required, then $\mathrm{ReLU}^2(x)/2$ would not be allowed.
\end{remark}

\section{Numerical experiments}
\label{sec:numerical experiments}

This section evaluates the gradient-descent-type nonexpansive architectures we propose in two complementary regimes. The first experiment is a supervised classification problem on the Poincar\'e disk $\mathbb{D}^2$. It is designed to test whether the constrained Busemann layers can fit non-linear hyperbolic decision boundaries while being robust under adversarial perturbations. The second experiment is an inverse problem on $S_{++}(10)$, the manifold of $10\times 10$ symmetric positive definite matrices, where a trained Busemann denoiser is inserted into a Plug-and-Play (PnP) reconstruction loop for masked Wishart observations. This second experiment tests whether the nonexpansive denoiser can act as an effective learned prior beyond a likelihood-only reconstruction.

In both cases, we compare against baseline models that preserve the underlying geometry but lack Lipschitz constraints. We provide more details on each of the two problems and the considered architectures in the dedicated subsections below. The code accompanying the paper can be found at the GitHub repository \href{https://github.com/davidemurari/one-lipschitz-hadamard-networks}{https://github.com/davidemurari/one-lipschitz-hadamard-networks}.

\subsection{Hyperbolic classification on \texorpdfstring{$\mathbb{D}^2$}{D2}}
\label{subsec:hyp-classification}

We use two synthetic datasets on the Poincar\'e disk $\mathbb{D}^2$. We sample the points first in the Euclidean plane and then map them to $\mathbb{D}^2$ through $\exp_0$. The first class is sampled uniformly in $B_{0.45,\ell^2}(0)$. The second
class is sampled in
$B_{0.95,\ell^2}(0)\setminus B_{0.62,\ell^2}(0)$, with radius distributed as a Gaussian centered at $0.78$ and standard deviation
$0.08$, clipped to $[0.62,0.95]$. Here, we write $B_{r,\ell^2}(0)$ for the closed Euclidean ball in $\mathbb{R}^2$ of radius $r>0$ and center at the origin. The second dataset consists of $k=12$ angular classes arranged in radial sectors. Half of the classes concentrate near the Euclidean disk radius $r_{\mathrm{in}}=0.7$ and the remaining classes near $r_{\mathrm{out}}=0.8$. Each point is obtained by perturbing the associated reference radius and angle with Gaussian additive noise with zero mean and fixed radial standard deviation $\sigma_{\mathrm{rad}}=0.02$ and angular standard deviation $\sigma_{\mathrm{ang}}=0.16$. Both datasets contain $4000$ points, split into an 80/20 train--test split. These tasks are intentionally simple enough to visualize, and they allow us to inspect the mathematical behavior of the models we are considering. We show the plot of the two datasets in \Cref{fig:dataset_classification}.
\begin{figure}[h!]
    \centering
    \includegraphics[width=0.3\linewidth]{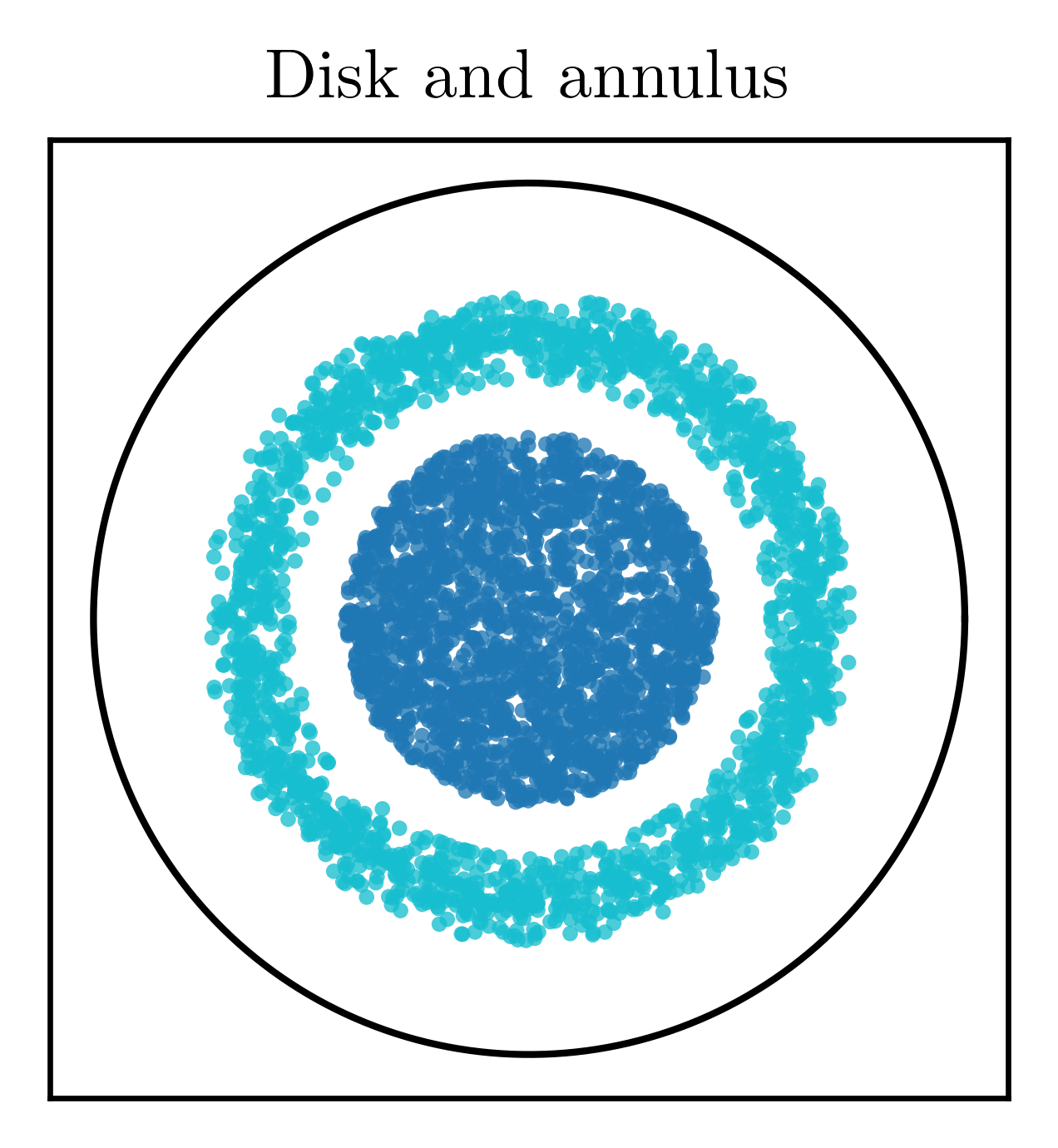}
    \includegraphics[width=0.3\linewidth]{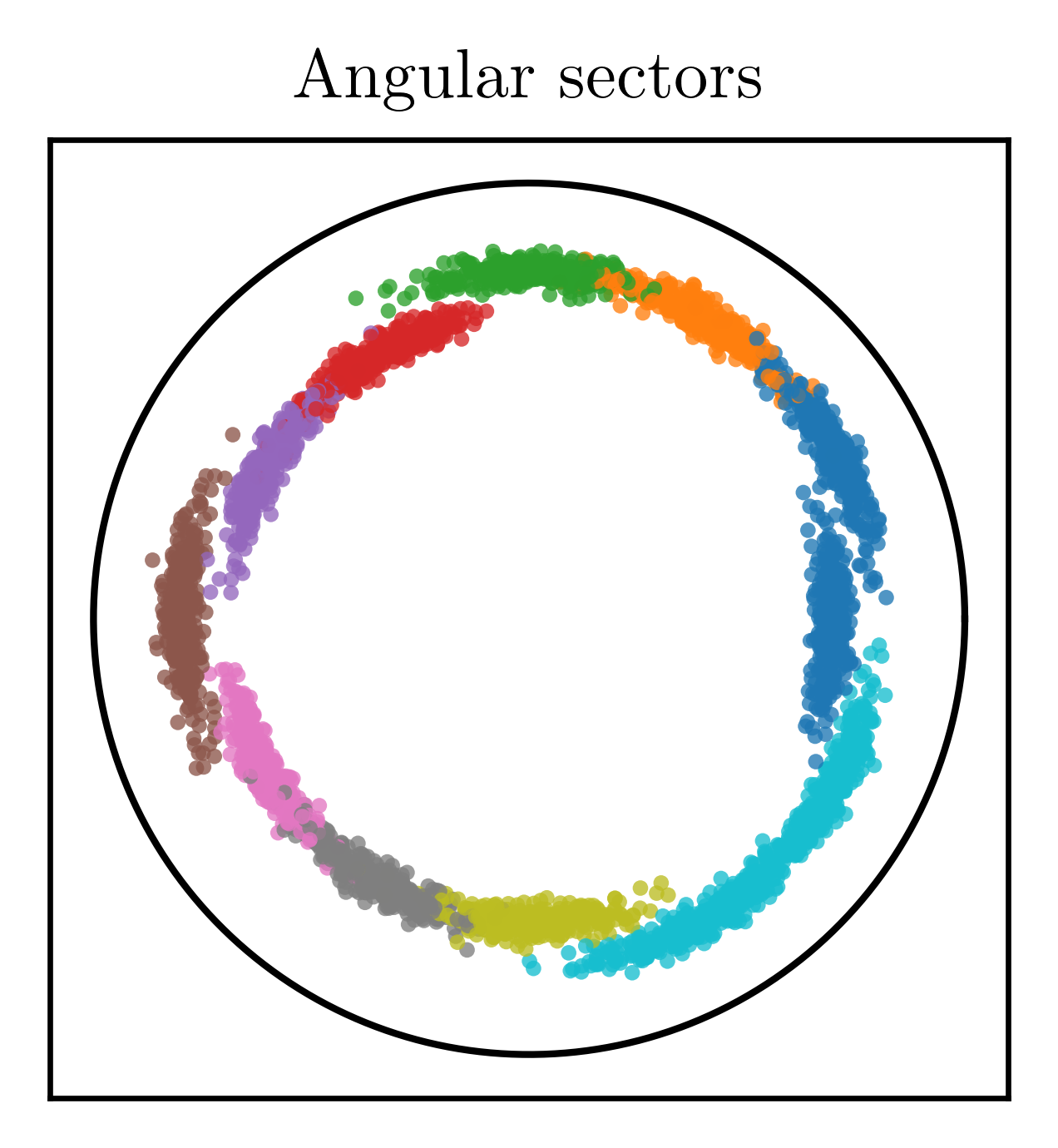}
    \caption{Datasets used for the classification problem.}
    \label{fig:dataset_classification}
\end{figure}

\begin{algorithm}[h!]
\caption{Schematic comparison of the two main hyperbolic classifiers.}
\label{alg:hyp-architecture-comparison}
\begin{algorithmic}[1]
\REQUIRE Input $x\in\mathbb{D}^2$
\STATE
\begin{tabular}{p{0.47\linewidth}p{0.47\linewidth}}
\textbf{Nonexpansive Busemann model} &
\textbf{Hyperbolic ResNet baseline}
\\[0.4em]

$z \gets (x,0)\in\mathbb{D}^3$
&
$z \gets (x,0)\in\mathbb{D}^3$
\\
\emph{trivial isometric embedding}
&
\emph{trivial isometric embedding}
\\[0.5em]

\textbf{for} $r=1,2$ \textbf{do}
&
\textbf{for} $r=1,2$ \textbf{do}
\\[0.3em]

\quad $z \gets A_r(z)$
&
\quad $z \gets A_r(z)$
\\
\quad \emph{Learnable isometry on $\mathbb{D}^3$}
&
\quad \emph{Learnable M\"{o}bius affine map}
\\[0.5em]

\quad \textbf{for} $j=1,\dots,5$ \textbf{do}
&\quad $z \gets \exp_z(f_r(z))$
\\[0.3em]

\qquad $z \gets G_{b_{r,j}}(z)$
&  \quad \emph{unconstrained residual step \eqref{eq:hypResidual}}
\quad 
\\
\qquad \emph{Busemann step, \Cref{alg:hyp-single-busemann}}
&
\\[0.3em]

\quad \textbf{end for}
&
\\[0.3em]

\textbf{end for}
&
\textbf{end for}
\\[0.5em]

$z \gets A_{3}(z)$ & $z \gets A_{3}(z)$
\\
\emph{Learnable isometry on $\mathbb{D}^3$}
&
\emph{Learnable M\"{o}bius affine map} \\
[0.5em]

$\ell \gets H(z)\in\mathbb{R}^{C}$
&
$\ell \gets H(z)\in\mathbb{R}^{C}$
\\
\emph{classification head \eqref{eq:proto}}
&
\emph{classification head \eqref{eq:proto}}
\\[0.5em]

\textbf{return} $\ell$
&
\textbf{return} $\ell$
\end{tabular}
\end{algorithmic}
\end{algorithm}

\paragraph{Architectures}
We consider three models for this task. Two are 1-Lipschitz before the classification head, and the other is unconstrained. We now provide a precise description of how they are assembled. An algorithmic description of the Busemann and ResNet architectures is provided in \Cref{alg:hyp-architecture-comparison}. The isometric model follows a similar logic.

All the models start with a trivial embedding of the data points from $\mathbb{D}^2$ to $\mathbb{D}^3$ defined as $x\mapsto (x,0)$. $\mathbb{D}^3$ is the manifold where the hidden maps are defined. The classifier head of all the models is a hyperbolic prototype head with $C$ learned prototypes, where $C$ is the number of classes.  More precisely, we learn $C$ target points $t_1,\dots,t_{C}\in\mathbb{D}^{3}$ and return a vector in $\mathbb{R}^{C}$ defined as
\begin{equation}\label{eq:proto}
\mathbb{D}^3\ni h\mapsto \begin{bmatrix} -d(h,t_1) &\dots & -d(h,t_{C}) \end{bmatrix}^\top,
\end{equation}
where $h$ is the output of the last hidden layer. This head is designed so that the highest output entry corresponds to the closest prototype. 

\textbf{1-Lipschitz Busemann model.} The core constrained model uses a split Busemann feature map. We always use $\mathrm{ReLU}^2/2$ as the activation for the Busemann layers. Each hidden split block is a composition of single-Busemann updates as defined in \Cref{sec:hypArch}, and in the implementation every such block is preceded by its own learnable hyperbolic isometry. The isometries are defined as $\mathbb{D}^3\ni y\mapsto b\oplus_{\mathbb{D}^3} Qy\in\mathbb{D}^3$ where $Q\in O(3)$ is a learnable orthogonal matrix, $b\in \mathbb{D}^3$ is a learnable bias, and $\oplus_{\mathbb{D}^3}$ stands for the Möbius addition in $\mathbb{D}^3$, see \cite{ganea2018hyperbolic}. The learnable isometries are included because they preserve the network 1-Lipschitz structure and introduce additional flexibility into the model. The layer parameters are constrained so that the layer remains compatible with the nonexpansive construction. We consider two blocks of Busemann steps, each containing $5$ gradient-descent steps and one isometry.

\textbf{1-Lipschitz isometric model.} This model is constrained as well, and provides a 1-Lipschitz baseline useful to judge the complexity of the task. It is obtained using the same structure as the Busemann model, but replacing the blocks of Busemann gradient-descent steps with an additional learnable isometry of $\mathbb{D}^3$.

\textbf{Unconstrained hyperbolic ResNet.} As a second baseline, we use an unconstrained hyperbolic residual network with the same structure as the other two models. The isometric layers are replaced by unconstrained M\"obius affine maps, see \cite[Equations~(27)--(28)]{ganea2018hyperbolic}. The Busemann blocks are replaced by unconstrained Poincar\'e residual steps defined as
\begin{equation}\label{eq:hypResidual}
x\mapsto \exp_x\left(\tau \mathrm{PT}_{0\to x}(W_2\sigma (W_1\log_0 x + b_1)+b_2)\right),\,\,\tau>0,
\end{equation}
where $\mathrm{PT}_{0\to x}$ is the parallel transport from the origin to $x$, $\log_0$ and $\exp_x$ are the logarithmic and exponential maps at the origin and $x$, respectively, and $\sigma$ is a scalar nonlinearity. We use $\sigma=\mathrm{ReLU}$ and learn matrices $W_1,W_2\in\mathbb{R}^{3\times 3}$. 

\begin{algorithm}[t]
\caption{One hyperbolic single-Busemann step}
\label{alg:hyp-single-busemann}
\begin{algorithmic}[1]
\REQUIRE $x\in\mathbb{D}^d$, raw direction $p_{\mathrm{raw}}$, parameters $\lambda,\beta,\tau$, scalar potential $\varphi$
\STATE $p \gets p_{\mathrm{raw}}/\|p_{\mathrm{raw}}\|_2$
\STATE $b \gets b_p(x)=\log \frac{\|x-p\|^2_2}{1-\|x\|^2_2}$
\STATE $g \gets \nabla b_p(x)$
\STATE $s \gets \tau\,\lambda\,\varphi'(\lambda b+\beta)$
\RETURN $\exp_x(-s g)$
\end{algorithmic}
\end{algorithm}

\paragraph{Training and latent dynamics}
All models are trained by minimizing the multiclass cross-entropy loss, which, on one batch $\mathcal{B}$ of size $|\mathcal{B}|$, writes
\[
    \mathcal{L}(\theta)=\frac{1}{|\mathcal{B}|}\sum_{(x_i,y_i)\in\mathcal{B}}\mathrm{CE}\!\left(f_\theta(x_i),y_i\right),
\]
where $f_\theta(x)$ denotes the vector of class scores obtained through the classifier head. We train all the models for $200$ epochs with Adam, batch size $256$, a one-cycle learning-rate schedule from $5\cdot 10^{-4}$ to $5\cdot 10^{-3}$, no weight decay, and a validation set consisting of $10\%$ of the training set. We also clip the gradient norm to $1$ while training. For each model seed, the checkpoint with the best validation accuracy is retained and then evaluated on the held-out test set. The reported comparison uses the three model seeds $7,11,17$. We report in \Cref{fig:decisionRegion} the ensemble decision regions of the three models over the two classification problems.
\begin{figure}[h!]
    \centering
    \includegraphics[width=0.6\linewidth]{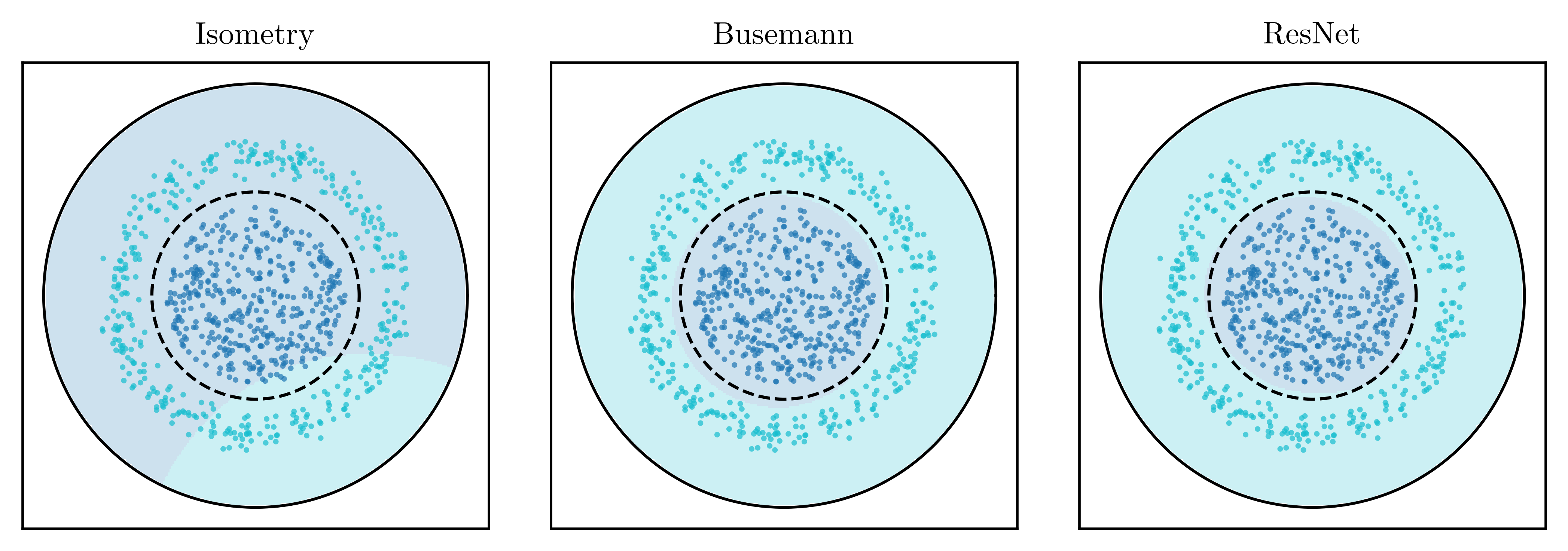}

    \includegraphics[width=0.6\linewidth]{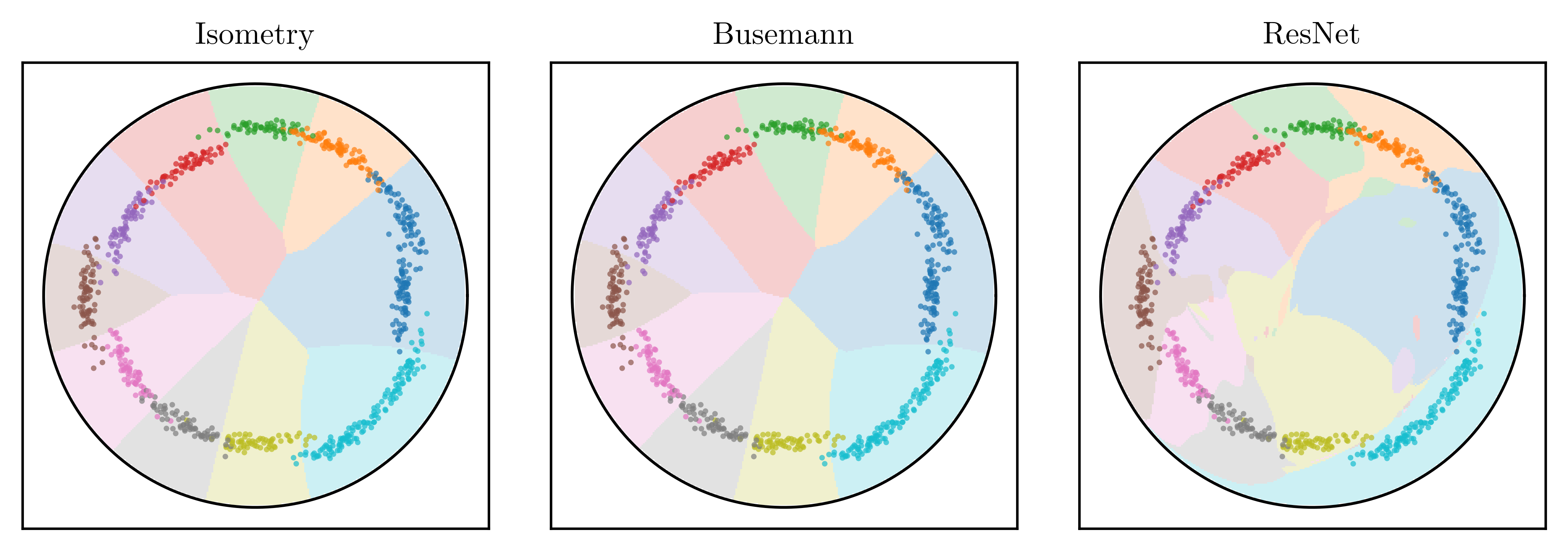}
    \caption{For each model and each dataset, we have three seeds. We present the decision regions derived from the probability vector obtained by averaging the predicted probabilities across the three seeds. To extract probabilities from the logits, we use the $\mathrm{softmax}$ function. For the annular dataset, we include the circular separatrix obtained as $\exp_0(\partial B_{(r_1+r_2)/2,\ell^2}(0))$, which is the ideal decision boundary given our dataset.}
    \label{fig:decisionRegion}
\end{figure}

For each model trained with seed $7$, \Cref{fig:dynamics} shows the evolution of the data points across the successive feature map layers. The plots make the different geometric mechanisms explicit. The isometric model can only relocate the data by constrained hyperbolic isometries. The Busemann model progressively pulls selected regions through nonexpansive Busemann-gradient steps, and the ResNet baseline uses unconstrained residual deformations. The end figure for each model represents, with squares, the learned prototype points $t_1,t_2\in\mathbb{D}^3$ (see \eqref{eq:proto}).

\begin{figure}[h!]
    \centering
    \includegraphics[width=.8\linewidth]{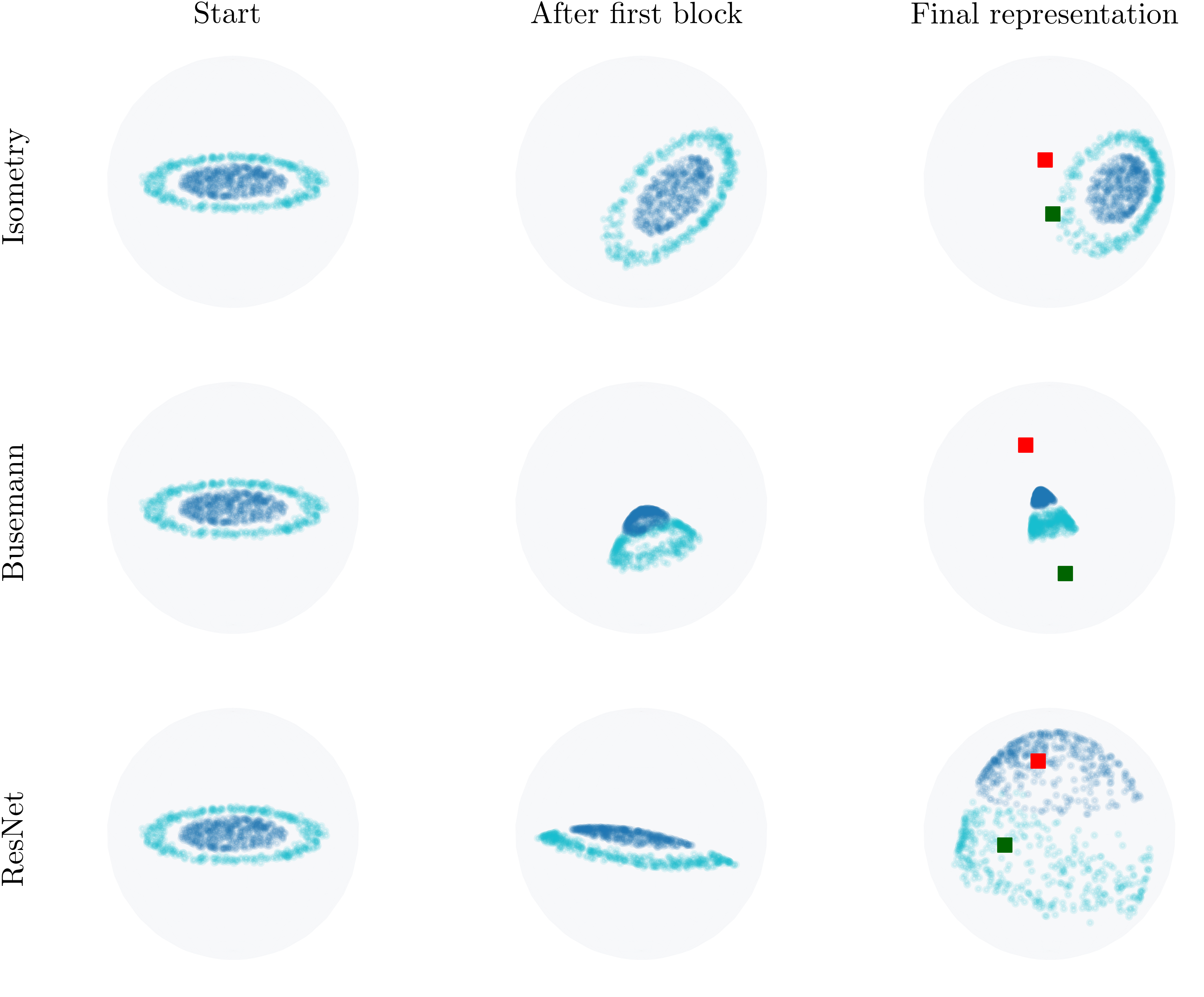}
    \caption{Layer-wise evolution of the annular test data through the feature maps for model seed $7$. Rows correspond to the Isometry, Busemann, and ResNet models, while columns show the initial embedding, the representation after the first block, and the final representation. The Isometry model can only move the data via hyperbolic isometries, whereas the Busemann model performs structured nonexpansive deformations through compositions of Busemann-gradient steps. The ResNet baseline is unconstrained and can produce sharper residual deformations. The square markers in the last column show the pairs of learned prototypes.}
    \label{fig:dynamics}
\end{figure}

\paragraph{Adversarial attacks and robustness evaluation}
Adversarial examples are generated by Riemannian projected gradient ascent in geodesic balls around each test point, using the projected-gradient attack paradigm and the multiclass cross-entropy loss. Each ascent step maps the Euclidean gradient to the Poincar\'e Riemannian gradient, moves with the exponential map, and projects back to the geodesic ball if necessary. For each test point, we run the projected-gradient ascent procedure several times from independently sampled initial points in the geodesic ball $B(x_0,\varepsilon):=\{x\in\mathbb{D}^2:\,\,d(x_0,x)\leq \varepsilon\}$. We refer to these independent initializations as restarts. A point is counted as robustly classified at radius $\varepsilon$ if all restarts fail to find a misclassified point in $B(x_0,\varepsilon)$. We use perturbation radii $\varepsilon \in \{0, 0.1, 0.2, 0.3, 0.4\}$, with $200$ attack iterations, $20$ random restarts, and stepsize $3\varepsilon/200$. The robustness curves in \Cref{fig:hyp-robustness-curves} measure worst-case behavior under repeated local searches in hyperbolic geodesic balls.

\begin{figure}[h!]
    \centering
    \begin{subfigure}[t]{0.49\textwidth}
    \centering
    \includegraphics[width=.85\linewidth]{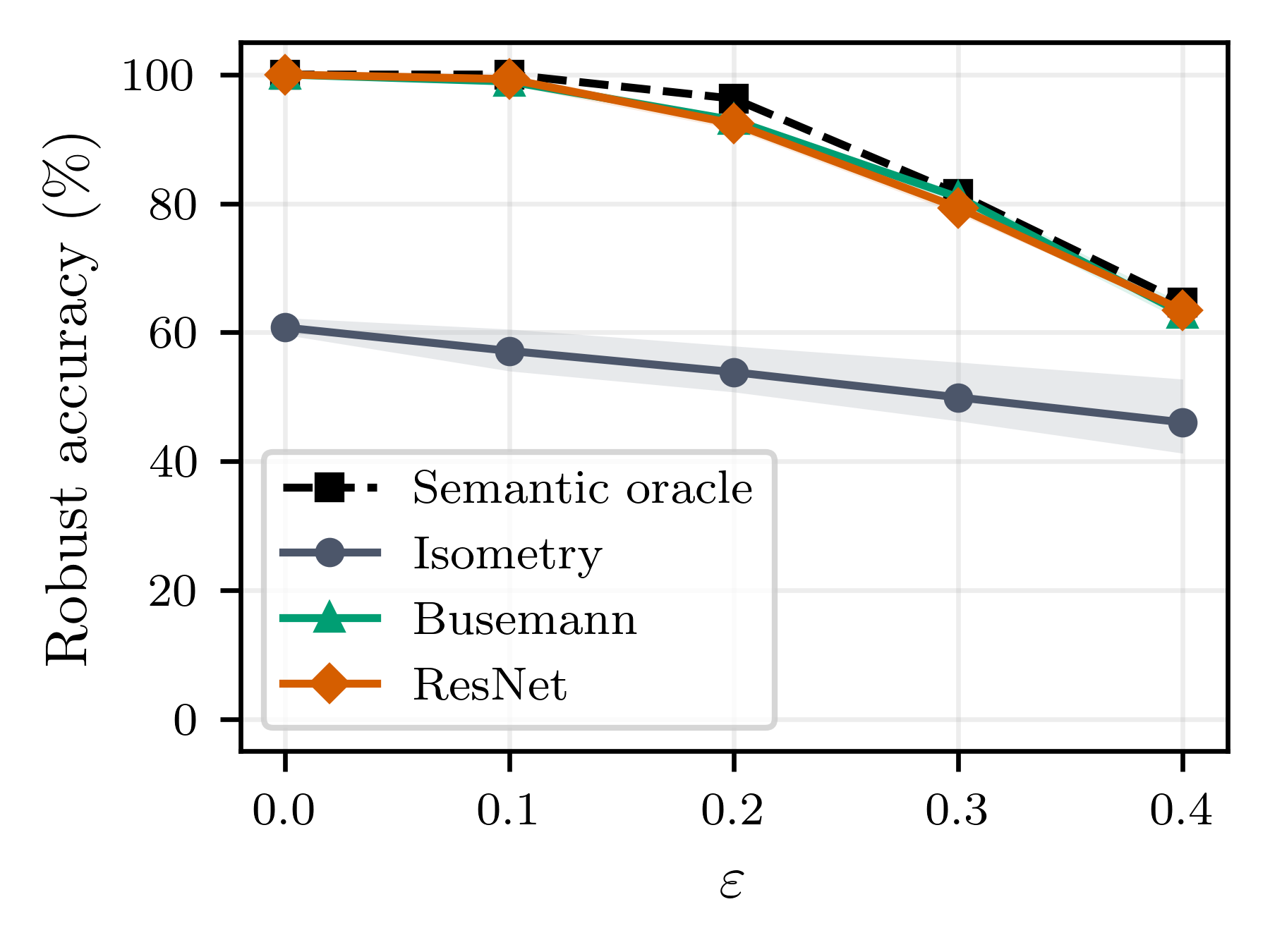}
    \caption{Annular dataset}
    \end{subfigure}
    \begin{subfigure}[t]{0.49\textwidth}
    \centering
    \includegraphics[width=.85\linewidth]{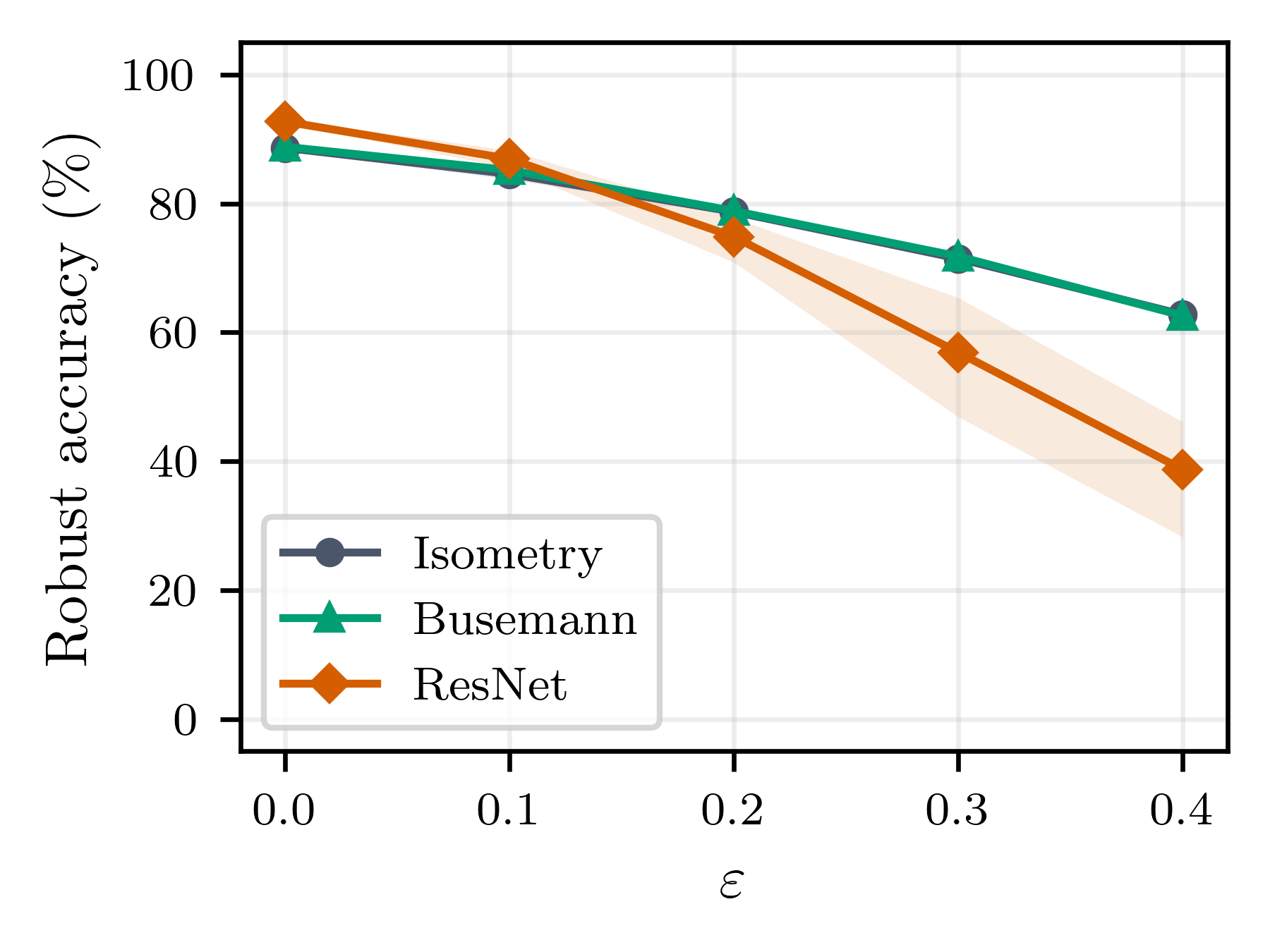}
    \caption{Angular dataset}
    \end{subfigure}
    \caption{Robust accuracy as the adversarial perturbation radius $\varepsilon$ increases. Solid curves show the mean over model seeds $7,11,17$, and the shaded bands show the corresponding min--max envelope across the three seeds. Semantic oracle denotes the reference circle representing the ideal classification boundary in the annular dataset.}
    \label{fig:hyp-robustness-curves}
\end{figure}

\paragraph{Interpretation}
The two datasets highlight complementary aspects of the constructions. On the annular task, the isometric baseline is not expressive enough: its mean clean accuracy is only about $60.8\%$. In contrast, both the Busemann model and the ResNet reach essentially perfect clean accuracy. Their robustness curves are also close to the ideal classifier provided by the intermediate circle, which we call \textit{oracle}. Their normalized Areas Under the Curves (AUCs)\footnote{This is a common quantity used to measure the tradeoff between accuracy and robustness. It is the area under the curve computed using the trapezoidal formula, and divided by $0.4$, the total length of the epsilon range.} are $88.6$ for the Busemann model, $88.2$ for the ResNet, and $90.0$ for the oracle. Thus, in this task, the Busemann layers add geometric expressivity beyond isometries while preserving the nonexpansive structure. In other words, the Busemann layers play a fundamental role in achieving the final test accuracy.

The angular-sector task gives a different message. Here the isometric baseline is already competitive with the Busemann model, indicating that this dataset can largely be solved by moving the prototype geometry with hyperbolic isometries. Nevertheless, the constrained models are more stable than the unconstrained ResNet under attack. The ResNet obtains higher clean accuracy, about $92.7\%$, but its robust accuracy drops to about $38.7\%$ at $\varepsilon=0.4$. The constrained models have lower clean accuracy, about $88.7\%$, but retain about $62.6\%$ robust accuracy at the same radius. This supports the intended role of the nonexpansive constraint. It does not by itself guarantee higher clean accuracy, and it does not always make the Busemann model more expressive than isometries. However, it biases the learned feature map toward more stable geometry under hyperbolic perturbations.

\subsection{Masked-Wishart inverse problem on \texorpdfstring{$S_{++}(10)$}{S_{++}(10)}}
\label{subsec:spd-inverse}

The second experiment studies the reconstruction of an unknown covariance matrix $X_\star\in S_{++}(n)$ from noisy masked principal-block observations. We consider the case $n=10$. Each of the $L$ masks $\ell=1,\dots,L$ selects a coordinate set
$I_\ell\subset\{1,\dots,n\}$ of $p=|I_\ell|$ coordinates. We denote by
$P_\ell\in\mathbb{R}^{p\times n}$ the corresponding coordinate-selection
matrix, so that $P_\ell x=x_{I_\ell}\in\mathbb{R}^p$. If
$x\sim\mathcal{N}(0,X_\star)$, then the observed subvector $P_\ell x$ is
Gaussian with covariance
\[
    C_\ell(X_\star)=P_\ell X_\star P_\ell^\top .
\]
For each mask, we draw $m=20$ independent samples
$y_{\ell,1},\dots,y_{\ell,m}\sim
\mathcal{N}(0,C_\ell(X_\star))$ and observe the empirical covariance
\[
    S_\ell
    =
    \frac{1}{m}
    \sum_{r=1}^{m} y_{\ell,r}y_{\ell,r}^\top .
\]
Equivalently, $mS_\ell$ follows the Wishart law
$W_p(C_\ell(X_\star),m)$
\cite{wishart1928generalised,muirhead1982aspects}. For a candidate covariance
$X$, the likelihood measures how compatible the predicted masked covariance
$C_\ell(X)$ is with the observed empirical covariance $S_\ell$. Since \(mS_\ell\sim W_p(C_\ell(X),m)\), the negative log-likelihood
of the observed sample covariances, up to additive constants independent
of \(X\), is
\begin{equation}\label{eq:negLog}
    F(X)
    = \frac{m}{2}\sum_\ell
      \left(
      \log\det C_\ell(X)
      + \operatorname{tr}\!\left(C_\ell(X)^{-1}S_\ell\right)
      \right).
\end{equation}
This objective is
the data-consistency term. It favors covariances whose observed principal
blocks explain the empirical covariances $S_\ell$.

The masks are fixed overlapping principal blocks. In the reported experiment, we set $n=10$, block size $p=5$, and $L=3$, with coordinate sets
\[
I_1=\{1,\dots,5\},\qquad
I_2=\{3,\dots,7\},\qquad
I_3=\{6,\dots,10\}.
\]
Thus $P_\ell X P_\ell^\top$ extracts a $5\times 5$ principal submatrix of
$X$. The induced observed-entry pattern is shown in \Cref{fig:spd-mask-pattern}.

The target covariances are sampled from a synthetic AR(1)-ensemble family. We
use the stationary Gaussian AR(1) covariance model
\cite{brockwell1991time}: for autocorrelation parameter $\rho$, the
covariance matrix has Toeplitz form
\[
    \Sigma(\rho)_{ij}=\rho^{|i-j|},\,\,i,j=1,\dots,n.
\]
Each target is generated by drawing
i.i.d. $\rho_1,\rho_2,\rho_3\sim\mathcal{U}(0.2,0.95)$ and setting
\[
    X_\star=\frac{1}{3}\sum_{q=1}^3 \Sigma(\rho_q).
\]
Since only masked blocks
are observed and the observations are noisy, Riemannian gradient descent with objective \eqref{eq:negLog} alone does not
encode the global covariance structure of the AR(1)-ensemble family. We
therefore solve the reconstruction problem with a Plug-and-Play (PnP) scheme that
alternates an affine-invariant gradient step for $F$ with a learned
SPD-valued denoiser, described below.
\begin{figure}[t]
    \centering
    \includegraphics[width=.25\linewidth]{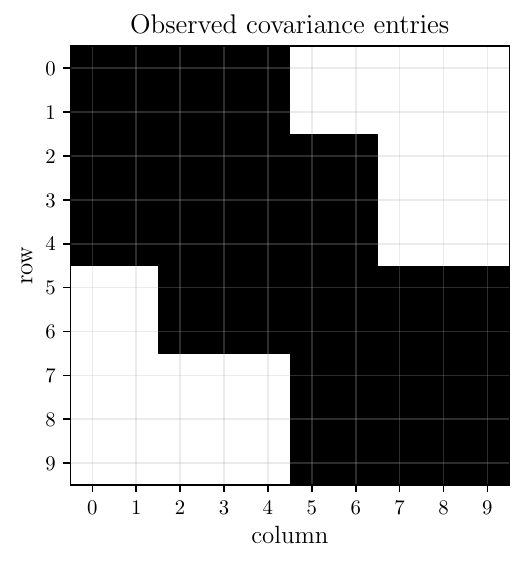}
    \caption{Observed-entry pattern induced by the three overlapping principal
    blocks. Colored cells are entries that appear in at least one observed
    $5\times5$ block.}
    \label{fig:spd-mask-pattern}
\end{figure}

\paragraph{Denoisers} We train two SPD-valued denoisers on full Wishart sample-covariance
observations generated from the same AR(1)-ensemble family. The supervised
training set contains $500$ clean covariances. For each one, a full Wishart
sample covariance is generated and used as the noisy input. The training task
maps this full noisy covariance estimate to its clean covariance target. More precisely, for each training target $X_{\star}^{(j)}$, $j=1,\dots,500$,
we draw $m=20$ independent samples
$z_{j,1},\dots,z_{j,m}\sim\mathcal{N}(0,X_{\star}^{(j)})$ and form
\[
    S^{(j)}=\frac{1}{m}\sum_{i=1}^m z_{j,i}z_{j,i}^\top .
\]
The denoiser is trained to map $S^{(j)}$ to $X_{\star}^{(j)}$.
The first model is the constrained split Busemann denoiser
$D_\theta:S_{++}(10)\to S_{++}(10)$. It comprises six
blocks, each consisting of a learnable congruence isometry
$X\mapsto QXQ^\top$, with $Q\in O(10)$, followed by nine
affine-invariant Busemann gradient steps. The Busemann potentials use
the scalar nonlinearity $\mathrm{ReLU}^2/2$. For more details, see \Cref{se:spd_def}. 
We enforce the stepsize constraint in \eqref{eq: condition for nonexpansiveness} to ensure that the
denoising map is nonexpansive in the affine-invariant metric.

The second model is a Log-Euclidean residual denoiser
\[
    D_\theta(X)
    =
    \exp\left(\log X+\operatorname{sym}(R_\theta(\log X))\right),
    \qquad
    \operatorname{sym}(A)=\frac{A+A^\top}{2}.
\]
Here $R_\theta$ is a feedforward ReLU network applied in logarithmic
coordinates. In the implementation, $\log X$ is flattened into an
$n^2$-dimensional vector, passed through the network, reshaped back into an
$n\times n$ matrix, and then symmetrized before applying the matrix
exponential. For $n=10$, the network has width $100$ and two affine-ReLU
hidden stages before the output layer. This baseline is SPD-valued by construction, but it is not constrained to be nonexpansive in the
affine-invariant geometry. 

Both denoisers are trained for $100$ epochs in
\texttt{float64}, then frozen for the Plug-and-Play reconstruction.

\paragraph{Plug-and-Play reconstruction}
Given an iterate $X_k$, the PnP update first applies one affine-invariant
likelihood step for the masked-Wishart objective,
\[
    Z_k=\exp_{X_k}\!\left(-\tau\grad F(X_k)\right),
\]
where the affine-invariant Riemannian gradient is computed by converting the closed-form expression of the Euclidean gradient:
\begin{align*}
\nabla F(X) &= \frac{m}{2}\sum_\ell P_\ell^\top\left(C_\ell(X)^{-1}-C_\ell(X)^{-1}S_\ell C_\ell(X)^{-1}\right)P_\ell,\\
\grad F(X)&=X \nabla F(X) X.
\end{align*}
The matrix $Z_k$ is then averaged with the denoiser output along the affine-invariant
geodesic,
\begin{equation}\label{eq:denoising_avg}
    X_{k+1}=Z_k\#_\alpha D_\theta(Z_k),
\end{equation}
where
\[
    A\#_\alpha B
    =
    A^{1/2}\bigl(A^{-1/2}BA^{-1/2}\bigr)^\alpha A^{1/2}.
\]
This is the weighted geometric mean or geodesic convex combination of positive definite matrices
\cite[Eq.~(6.18)]{bhatia07pdm}. If $\alpha=1$, the map relies entirely on the learned denoiser. We remark that when $D_\theta$ is nonexpansive, \Cref{trm: alpha averaged operators are quasi alpha FME} applies, and ensures that the map in \eqref{eq:denoising_avg} is quasi-$\alpha$-firmly nonexpansive for $\alpha\in (0,1)$, and hence iterating only such steps without the data-fidelity step would ensure convergence to a fixed point, in case it exists (see \Cref{corollary: convergence to fix point for quasi firmly nonexpansive operators}). This is confirmed empirically in \Cref{fig:spd-denoiser-residuals}. The full PnP loop additionally
contains the data step, so the convergence theory is not used as a proof for
the inverse-problem iteration. Still, the empirical results support the validity of the proposed strategy.

We use $32$ validation observation-target
pairs and $200$ held-out test pairs using the masked-Wishart model above.
The parameters $\tau$, $\alpha$, and the stopping iterate $T$ are selected
on validation by minimizing mean $d_{\mathrm{AI}}(X_T,X_\star)$, and then fixed
for test evaluation. \Cref{fig:spd-tuning-heatmaps} shows the validation
grid used to select the PnP parameters for the two denoisers.

\begin{figure}[htbp]
    \centering
    \begin{minipage}{.48\linewidth}
        \centering
        \includegraphics[width=\linewidth]{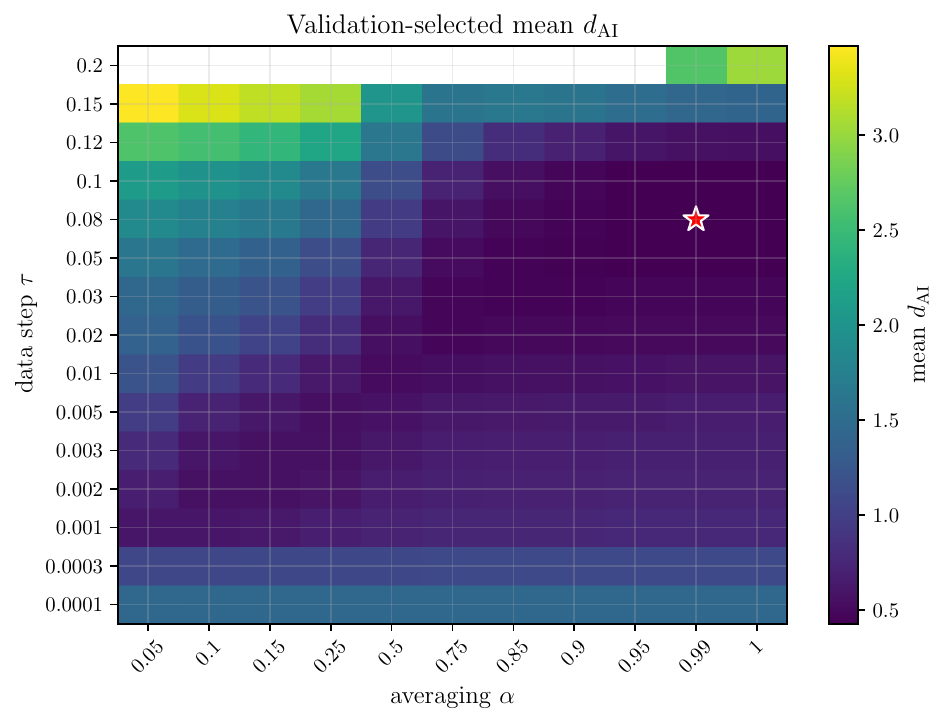}
        \smallskip

        \small Split Busemann
    \end{minipage}
    \hfill
    \begin{minipage}{.48\linewidth}
        \centering
        \includegraphics[width=\linewidth]{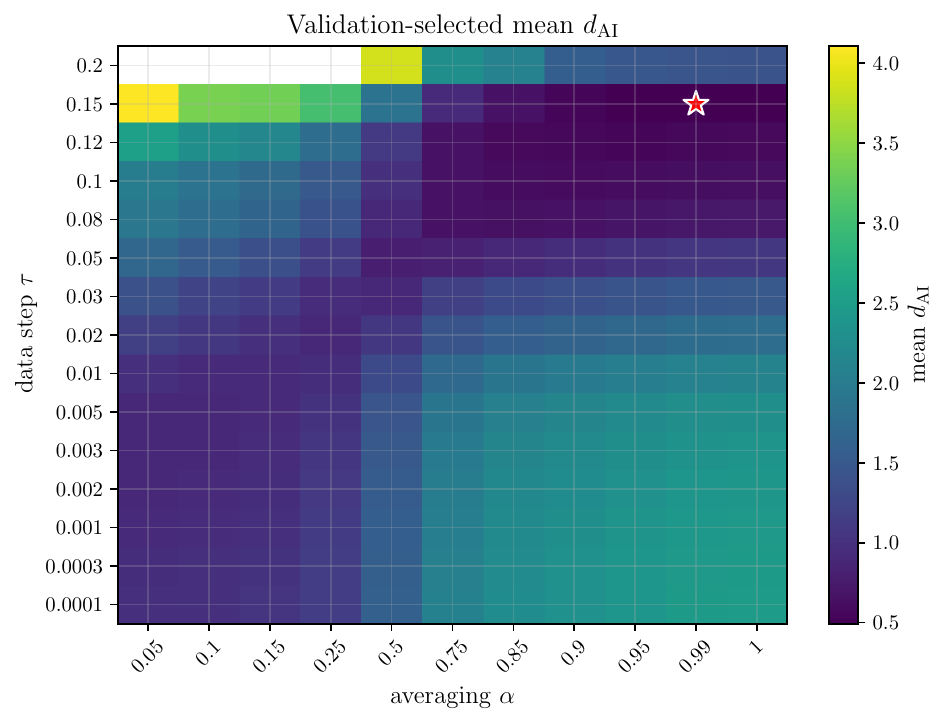}
        \smallskip

        \small Log-Euclidean
    \end{minipage}
    \caption{Validation tuning for the PnP reconstruction. For each denoiser,
    the full validation grid is run over the tested initializations, data steps
    $\tau$, averaging parameters $\alpha$, and stopping times. The heatmap then
    shows the $(\tau,\alpha)$ slice corresponding to the initialization with
    the best overall validation result for that denoiser. The star marks this
    global validation minimizer. \Cref{tab:spd-main-results} also reports
    the best validation-selected parameters separately for each listed
    initialization. White cells correspond to non-finite or failed runs, which
    occur only for overly aggressive data steps and are not selected.}
    \label{fig:spd-tuning-heatmaps}
\end{figure}

\paragraph{Baselines and metrics}
In addition to comparing the two denoisers, we benchmark different variations of the procedure. We use two types of baselines. First, we report static reconstructions that do
not use the masked observations: the identity matrix, the Euclidean and Log-Euclidean means of the $500$
training covariances used to train the denoisers.
Second, we report a data-only baseline, obtained by affine-invariant gradient
descent on the masked-Wishart objective $F$ without any denoising step. The likelihood-only baseline selects its own data stepsize and stopping time using the same validation protocol as the PnP methods. The primary metric is the affine-invariant distance $d_{\mathrm{AI}}(X,X_\star)$. We also record the masked-Wishart objective $F(X)$, log-Euclidean error, relative Frobenius error, observed-entry error, and unobserved-entry error. Qualitative examples are selected using relative Frobenius error, because those panels visualize entrywise matrix errors. All quantitative tables and tuning decisions use the affine invariant distance.

For the different procedures, we also test different initial guesses for the iterations. We start either from the identity matrix, the training Euclidean mean, or the training Log-Euclidean mean $\overline X_{\mathrm{LE}}=\exp\bigl(N_{\mathrm{tr}}^{-1}\sum_{i=1}^{N_{\mathrm{tr}}}\log X_i\bigr)$.

\begin{table}[htbp]
    \centering
    \small
    \begin{tabular}{lcccc}
        \hline
        method & initialization & $\tau$ & $\alpha$ & mean $d_{\mathrm{AI}}\downarrow$ \\
        \hline
        identity & identity & -- & -- & 2.873 \\
        Euclidean training mean & -- & -- & -- & 0.845 \\
        Log-Euclidean training mean & -- & -- & -- & 0.859 \\
        data-only descent & Euclidean mean & $10^{-4}$ & -- & 0.789 \\
        split Busemann PnP & identity & 0.08 & 0.99 & 0.369 \\
        split Busemann  PnP & Euclidean mean & 0.10 & 0.99 & \textbf{0.368} \\
        split Busemann  PnP & Log-Euclidean mean & 0.10 & 0.99 & \textbf{0.368} \\
        Log-Euclidean PnP & identity & 0.12 & 0.95 & 0.506 \\
        Log-Euclidean PnP & Euclidean mean & 0.15 & 1.00 & 0.480 \\
        Log-Euclidean PnP & Log-Euclidean mean & 0.15 & 0.99 & 0.475 \\
        \hline
    \end{tabular}
    \caption{Main test-set reconstruction results. All data-only and PnP
    parameters, including the stopping time, are selected based on the validation set and then
    frozen for test evaluation. The selected stopping times are $T=150$ for
    data-only descent from the Euclidean mean, $T=5$ for split PnP from
    identity, $T=50$ for split PnP from either training mean, and $T=10$
    for the Log-Euclidean PnP rows. We tune $\alpha\in(0,1]$, where $\alpha=1$ corresponds to applying the raw denoiser without geodesic relaxation. The averaged nonexpansive interpretation applies only when $\alpha\in(0,1)$ and the denoiser is nonexpansive.}
    \label{tab:spd-main-results}
\end{table}

\begin{table}[htbp]
    \centering
    \small
    \begin{tabular}{lccc}
        \hline
        PnP model & baseline & \makecell{mean\\improvement} $\uparrow$ & \makecell{fraction\\improved} $\uparrow$ \\
        \hline
        split Busemann & Euclidean training mean & 0.477 & 0.75 \\
        split Busemann & Log-Euclidean training mean & 0.491 & 0.78 \\
        split Busemann & data-only, Euclidean mean init & 0.422 & 0.855 \\
        Log-Euclidean & Euclidean training mean & 0.364 & 0.76 \\
        Log-Euclidean & Log-Euclidean training mean & 0.379 & 0.74 \\
        Log-Euclidean & data-only, Euclidean mean init & 0.309 & 0.78 \\
        \hline
    \end{tabular}
    \caption{Paired per-sample improvements on the test set. For each test
    covariance, we compute
    $d_{\mathrm{AI}}(X_{\mathrm{baseline}},X_\star)
    -d_{\mathrm{AI}}(X_{\mathrm{PnP}},X_\star)$ using the same ground-truth
    target $X_\star$. The ``mean improvement'' column averages this quantity
    over the test set. The ``fraction improved'' column is the fraction of test
    samples for which this paired difference is positive, i.e., for which the
    selected PnP reconstruction has smaller affine-invariant error than the baseline. For every row, $X_{\mathrm{PnP}}$ is obtained by initializing the corresponding PnP iteration at the Euclidean training mean; its parameters $\tau$, $\alpha$, and $T$ are those selected on validation for that initialization.}
    \label{tab:spd-paired-improvements}
\end{table}

\begin{figure}[htbp]
    \centering
    \begin{minipage}{.48\linewidth}
        \centering
        \includegraphics[width=\linewidth]{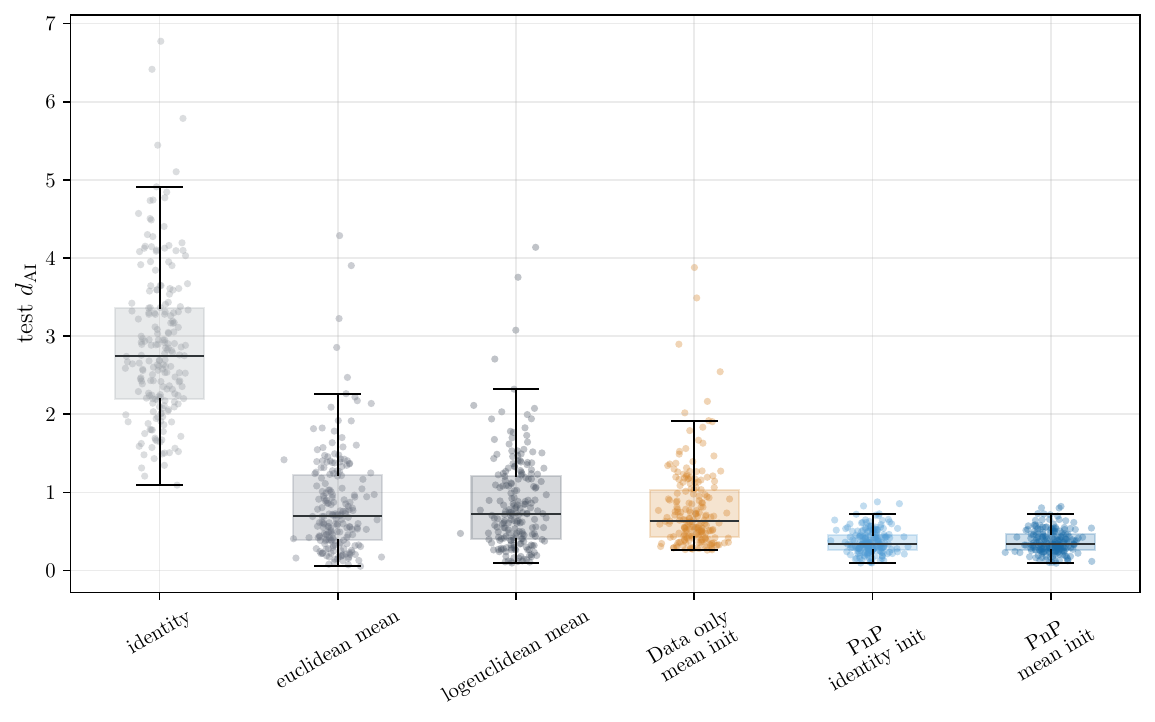}
        \smallskip

        \small Split Busemann
    \end{minipage}
    \hfill
    \begin{minipage}{.48\linewidth}
        \centering
        \includegraphics[width=\linewidth]{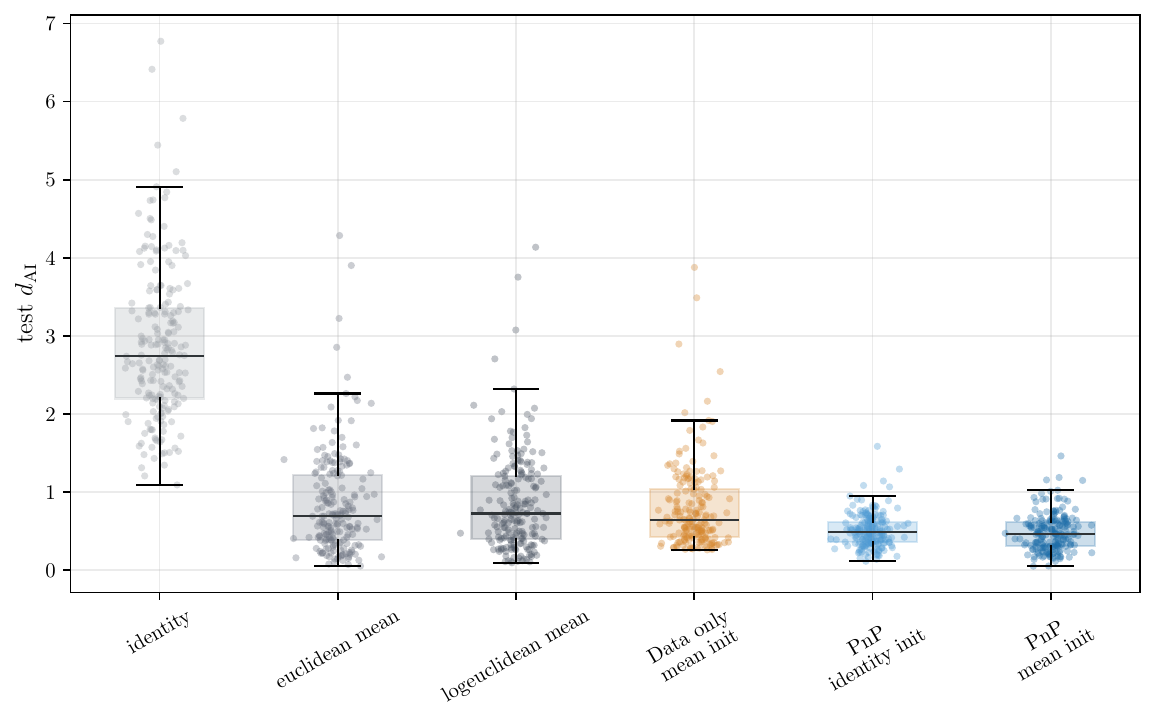}
        \smallskip

        \small Log-Euclidean
    \end{minipage}
    \caption{Per-sample test reconstruction errors corresponding to the methods
    and validation-selected configurations reported in
    \Cref{tab:spd-main-results}. For each denoiser and initialization, the PnP
    parameters $(\tau,\alpha,T)$ are selected independently on the validation set
    and then fixed for test evaluation; the data-only parameters $(\tau,T)$ are
    selected analogously. The boxplots summarize
    $d_{\mathrm{AI}}(X_T,X_\star)$ over the 200 held-out test instances, while the
    jittered points show the individual errors. The horizontal jitter is used
    only to improve visibility and has no quantitative meaning. The identity and
    training-mean entries are static baselines and require no parameter
    selection.}
    \label{fig:spd-final-distributions}
\end{figure}

\paragraph{Reconstruction results}
\Cref{tab:spd-main-results} shows that the training mean is already a
strong prior baseline for this concentrated AR(1)-ensemble. The Euclidean
training mean gives mean test error $0.845$, and data-only descent from that
mean improves it only to $0.789$. The split Busemann PnP reconstruction
reduces the error to about $0.368$ from either training-mean initialization. This shows that this denoiser is stable to initialization. We conjecture that this is a consequence of the denoiser's nonexpansiveness.
The Log-Euclidean denoiser also improves over data-only reconstruction, but its
best PnP row is $0.475$, which is weaker than the split Busemann
result. \Cref{fig:spd-final-distributions} shows the full test-error
distributions, while \Cref{tab:spd-paired-improvements} reports paired
improvements on the same test instances. The split model improves over the
data-only Euclidean-mean baseline on $85.5\%$ of test samples. \Cref{fig:spd-reconstruction-examples} gives representative correlation
matrices for a qualitative view of these reconstruction differences.

\begin{figure}[htbp]
    \centering
    \begin{minipage}{.48\linewidth}
        \centering
        \includegraphics[width=\linewidth]{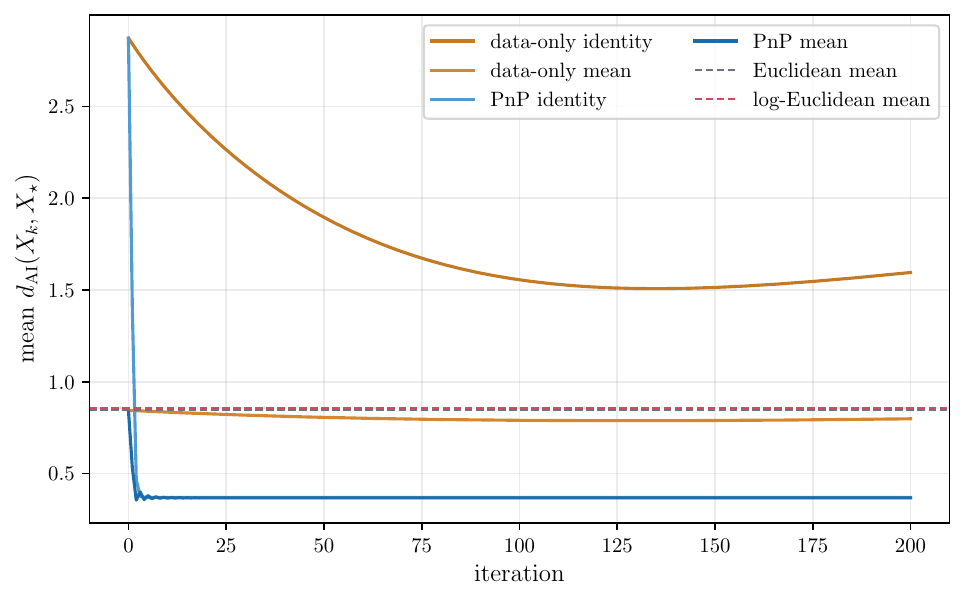}
        \smallskip

        \small Split Busemann
    \end{minipage}
    \hfill
    \begin{minipage}{.48\linewidth}
        \centering
        \includegraphics[width=\linewidth]{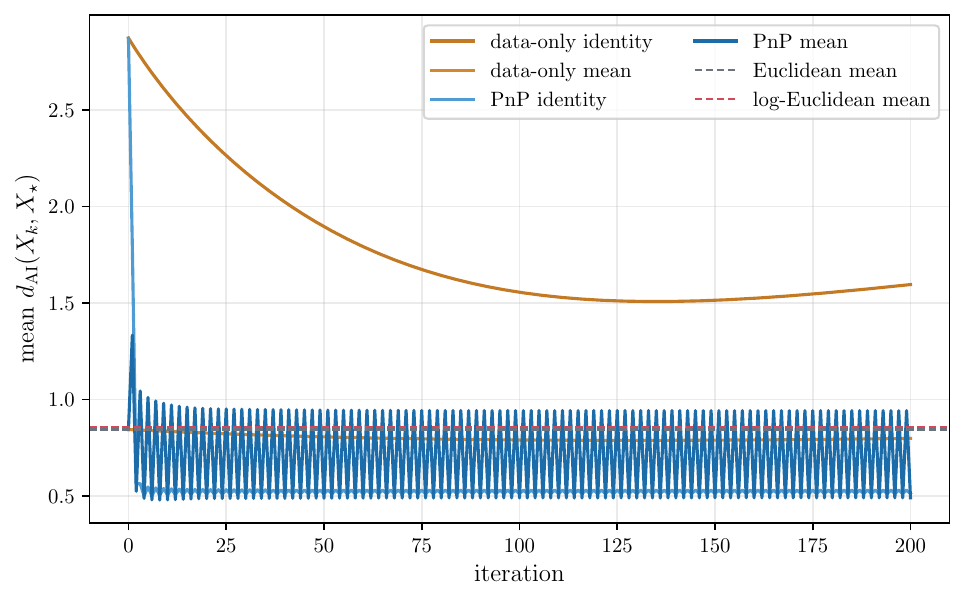}
        \smallskip

        \small Log-Euclidean
    \end{minipage}
    \caption{Mean target-error trajectories on the test set using the
    validation-selected stepsizes. At iteration $k$, the vertical axis is
    $N_{\mathrm{test}}^{-1}\sum_{b=1}^{N_{\mathrm{test}}}
    d_{\mathrm{AI}}(X_k^{(b)},X_\star^{(b)})$. These curves are diagnostic
    rollouts. The reported table values use validation-selected stopping times,
    while the plots in this figure show how the same dynamics behave when continued longer. The keywords \textit{identity} and \textit{mean} refer to the initial guess $X_0$, which is either set to the identity or the Euclidean mean, respectively.}
    \label{fig:spd-target-trajectories}
\end{figure}

\begin{figure}[htbp]
    \centering
    \begin{minipage}{.48\linewidth}
        \centering
        \includegraphics[width=\linewidth]{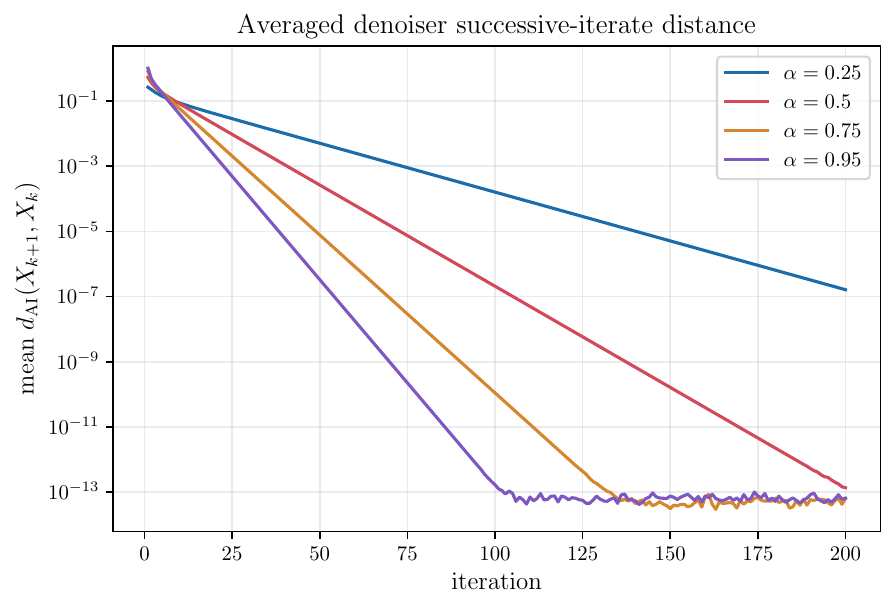}
        \smallskip

        \small Split Busemann
    \end{minipage}
    \hfill
    \begin{minipage}{.48\linewidth}
        \centering
        \includegraphics[width=\linewidth]{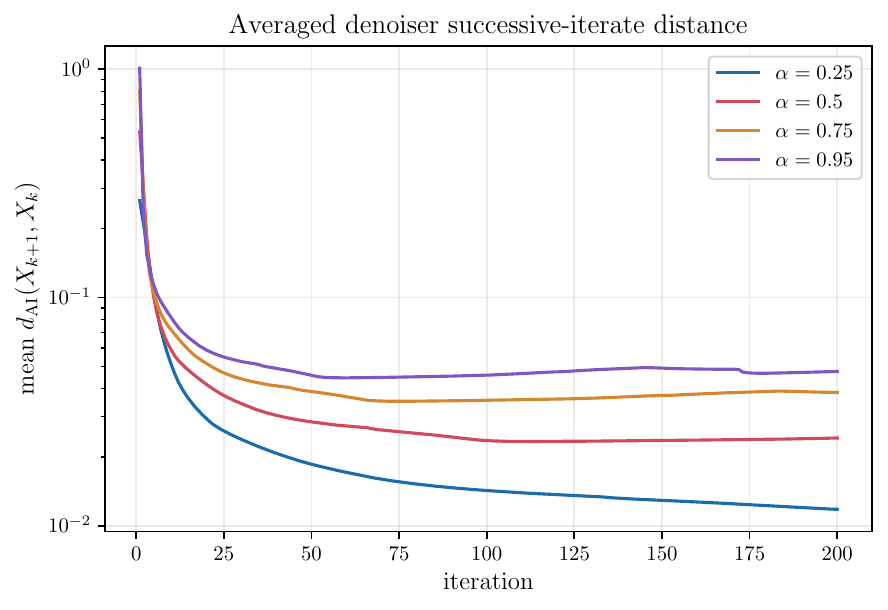}
        \smallskip

        \small Log-Euclidean
    \end{minipage}
    \caption{Averaged denoiser-only stability diagnostic. Each curve plots the
    mean successive-iterate distance
    $d_{\mathrm{AI}}(X_{k+1},X_k)$ for
    $X_{k+1}=X_k\#_\alpha D_\theta(X_k)$. This isolates the averaged denoiser
    from the masked-Wishart data step, which is the component covered by the
    nonexpansive fixed-point theory.}
    \label{fig:spd-denoiser-residuals}
\end{figure}

\paragraph{Stability interpretation}
\Cref{fig:spd-target-trajectories}
explains why early stopping is part of the reconstruction protocol. We can see that the validation-selected dynamics can initially reduce the test error but may deteriorate when continued beyond the selected stopping time. This is visible in the tuning heatmaps as failed cells for
large $\tau$, see \Cref{fig:spd-tuning-heatmaps}. The oscillations visible in the Log-Euclidean trajectory indicate that the
Log-Euclidean PnP dynamics are more sensitive to early stopping, whereas the
split Busemann dynamics remain more stable in this long-run diagnostic. \Cref{fig:spd-denoiser-residuals} isolates the denoiser-only averaged
map $X\mapsto X\#_\alpha D_\theta(X)$. This is the part of the method for
which the nonexpansive theory applies. In a Hadamard space, averaged
nonexpansive iterations have fixed-point convergence guarantees
when the fixed-point set is nonempty. The decay of
$d_{\mathrm{AI}}(X_{k+1},X_k)$ in the split-Busemann panel is therefore a
diagnostic support for the intended averaged-map behavior. It is still not a
convergence proof for the full PnP loop, because the full loop also contains
the masked-Wishart gradient step.

\begin{figure}[htbp]
    \centering
    \begin{minipage}{.48\linewidth}
        \centering
        \includegraphics[width=\linewidth]{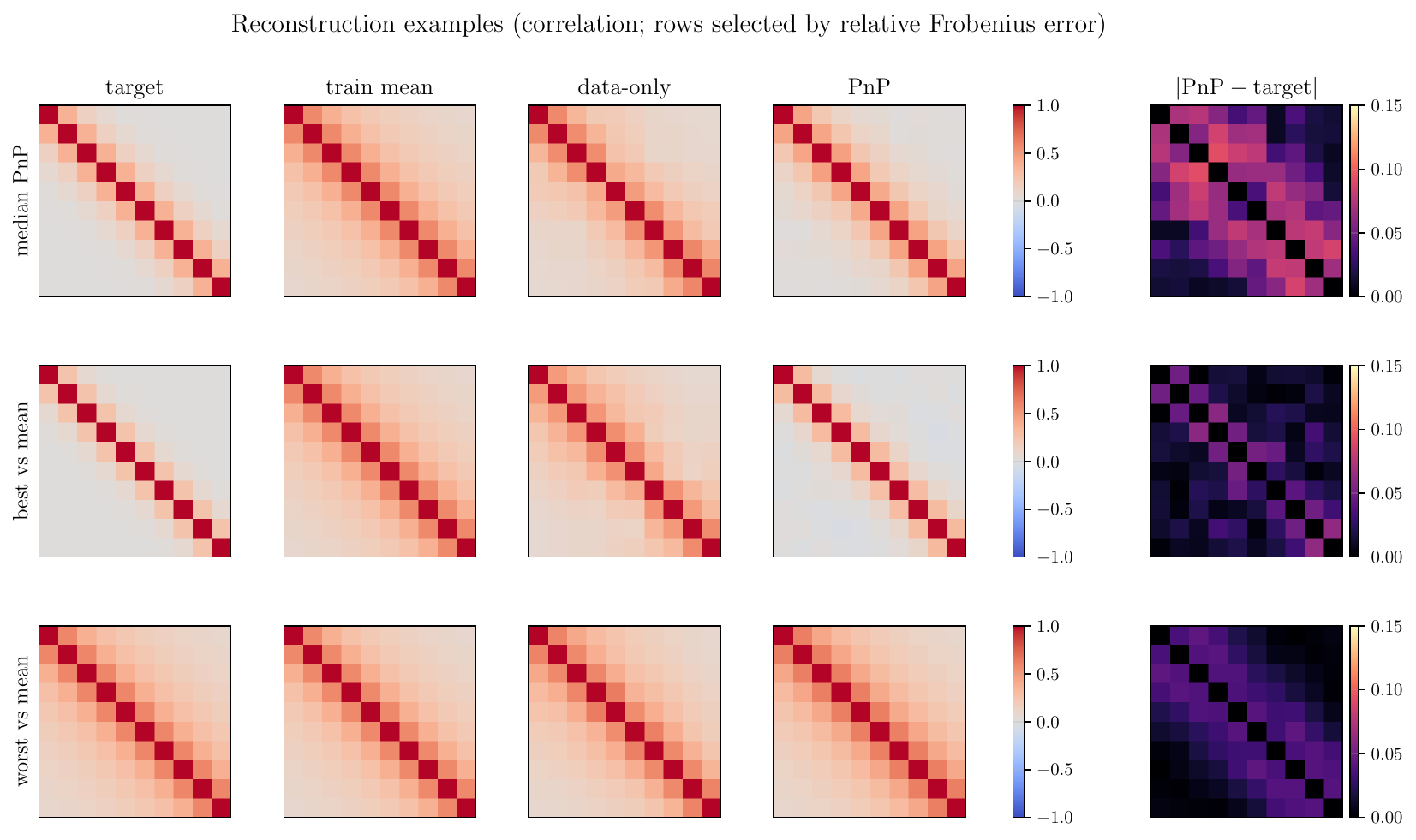}
        \smallskip

        \small Split Busemann
    \end{minipage}
    \hfill
    \begin{minipage}{.48\linewidth}
        \centering
        \includegraphics[width=\linewidth]{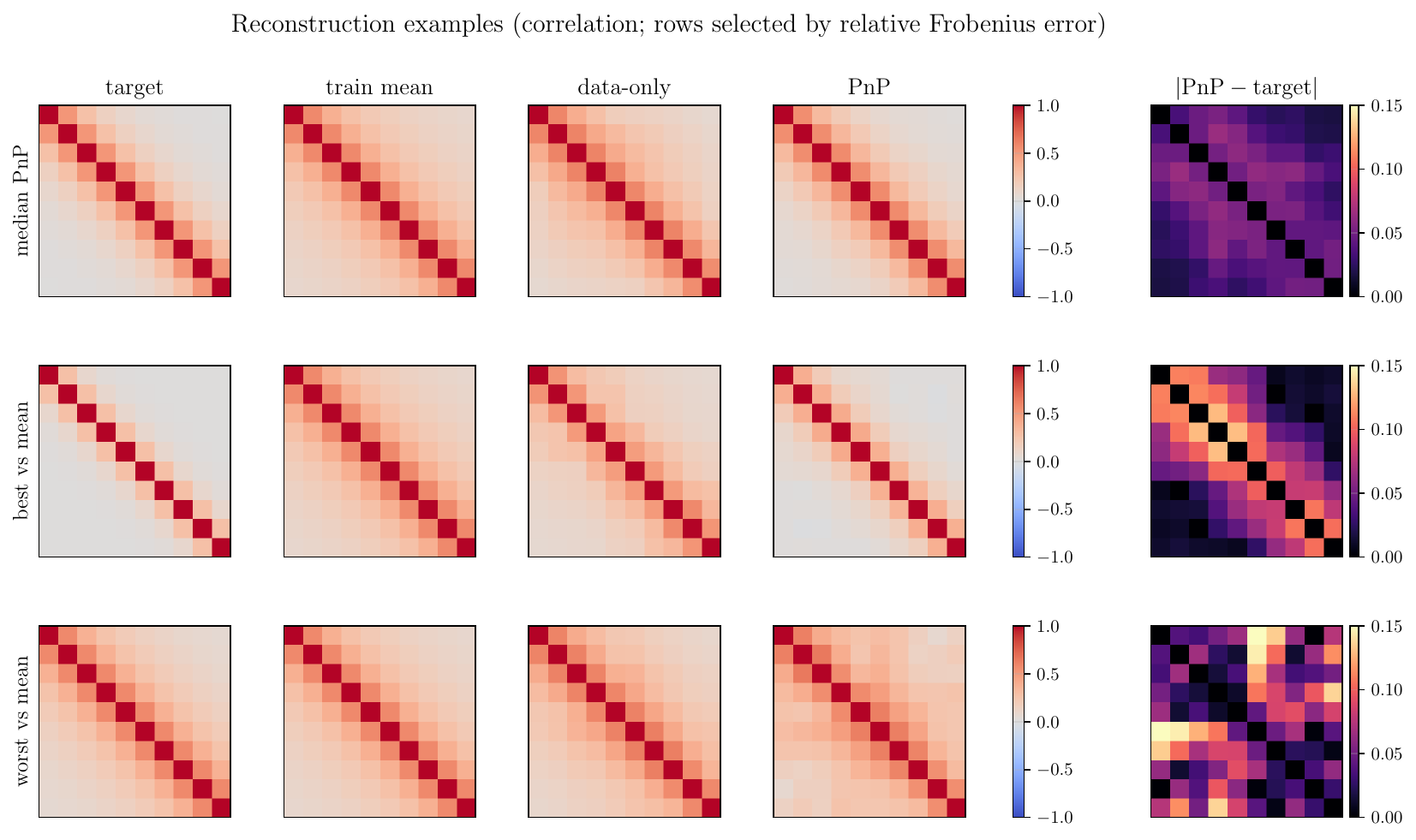}
        \smallskip

        \small Log-Euclidean
    \end{minipage}
    \caption{Representative correlation reconstructions. Rows are selected by
    relative Frobenius error to make the entrywise visual comparison
    interpretable: median PnP error, best PnP improvement over the training
    mean, and worst PnP change relative to the training mean. Error panels use a
    common color scale $[0,0.15]$ across both denoisers.}
    \label{fig:spd-reconstruction-examples}
\end{figure}

\begin{figure}[htbp]
    \centering
    \begin{minipage}{.48\linewidth}
        \centering
        \includegraphics[width=\linewidth]{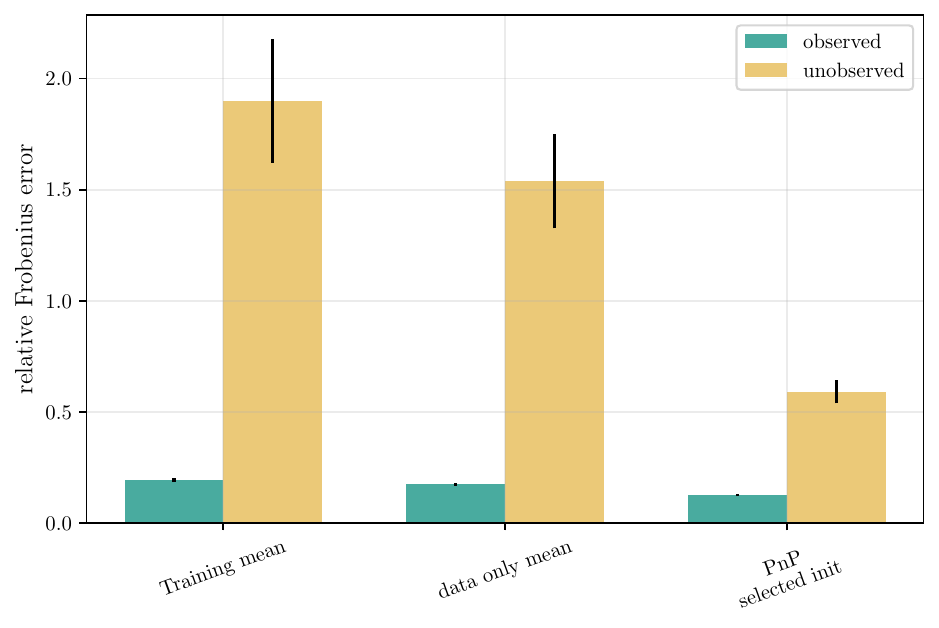}
        \smallskip

        \small Split Busemann
    \end{minipage}
    \hfill
    \begin{minipage}{.48\linewidth}
        \centering
        \includegraphics[width=\linewidth]{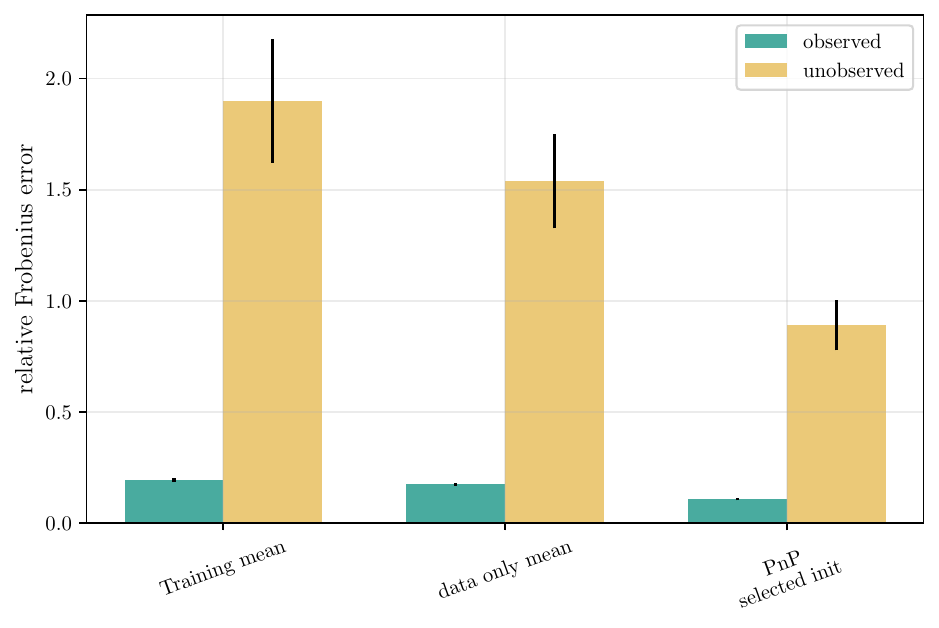}
        \smallskip

        \small Log-Euclidean
    \end{minipage}
    \caption{Observed-entry and unobserved-entry reconstruction errors for the
    two denoisers. The measured principal blocks determine only part of the
    covariance matrix. These diagnostics check whether the learned prior also
    improves entries not directly covered by the masks, rather than only
    fitting the observed blocks.}
    \label{fig:spd-extra-diagnostics}
\end{figure}

The observed/unobserved diagnostics in \Cref{fig:spd-extra-diagnostics} check that the improvement is not
limited to entries directly included in the measured principal blocks. This is important because the masks underdetermine the inverse problem. Therefore, good performance on unobserved entries indicates that the learned denoiser is adding
global covariance structure rather than only fitting the available blocks.

\paragraph{Conclusion for the SPD experiment}
The masked-Wishart experiment shows that a constrained split Busemann denoiser can be used as an effective learned prior in an SPD inverse problem. Under validation-selected PnP inference, the split Busemann denoiser substantially improves over both baselines and over the Log-Euclidean SPD-valued denoiser. The theory motivates the averaged denoising component, while the complete data-plus-denoiser reconstruction is evaluated empirically using validation tuning and held-out test instances.

\section{Conclusion}
\label{sec:conclusion}

In this paper, we introduced a new neural network architecture on Hadamard manifolds. This model is practically implementable for manifolds where a computationally efficient closed-form expression for Busemann functions is available. We theoretically analyzed such an architecture, proving that it is 1-Lipschitz when considering small enough step sizes in the manifold-valued residual layers. The derived stepsize restriction is cheap to evaluate and practically implementable. We also extended such a theory to product manifolds, providing a viable theoretical framework for future extensions to manifold-valued graphs and images. The proposed methodology was then empirically validated by two numerical experiments. The first considers the task of robustly classifying points in the Poincar\'e disk $\mathbb{D}^2$. The second considers an underdetermined covariance-matrix recovery problem.

The core focus of this paper has been on developing the methodology and comprehensively studying its theoretical aspects from the perspective of stability and convergence. Future extensions will consider image- and graph-focused applications, such as for diffusion tensor imaging. Our architecture relies on the concept of Busemann functions. A promising further extension of our methodology is to consider alternative building blocks which allow for a similar design principle in a broader class of Riemannian manifolds, extending the availability of 1-Lipschitz manifold-valued neural networks to a wider range of Riemannian manifolds.

\section*{Acknowledgements}
DM acknowledges support from the EPSRC programme grant in ‘The Mathematics of Deep Learning’, under the project EP/V026259/1.
CBS acknowledges support from the Royal Society Wolfson Fellowship, the EPSRC advanced career fellowship EP/V029428/1, the EPSRC programme grant EP/V026259/1, the Wellcome Innovator Awards 215733/Z/19/Z and 221633/Z/20/Z, the EPSRC funded ProbAI hub EP/Y028783/1. BA acknowledges support from the Natural Sciences and Engineering Research Council of Canada (NSERC) through grant RGPIN/2026-04531. All the authors acknowledge support from the EU through the Marie Sk\l{}odowska-Curie Actions Staff Exchanges project REMODEL, grant agreement 101131557.

\bibliographystyle{plain}
\bibliography{ManifoldNeuODE}

\newpage
\appendix

\section{Proofs omitted in the main document}
\label{appendix: proofs omitted in the main document}
We now provide a proof of \Cref{lemma: properties of Busemann functions on Hadamard manifolds} after restating it.
\begin{lemma*}
    Let $(\mathcal{M}, g)$ be a Hadamard manifold. Every Busemann function $b$ on $\M$ is geodesically convex, 1-Lipschitz, and of class $\C^{2}$. 
\end{lemma*}

\begin{proof}[Proof of \Cref{lemma: properties of Busemann functions on Hadamard manifolds}]
Let $\gamma : [0,+\infty) \to \mathcal{M}$ be a geodesic ray and $b_\gamma(x)$ its associated Busemann function given by \eqref{eq: busemann on manifold}. The characterization of horofunctions \cite[Proposition~8.22]{bridson2013metric} guarantees that the function $b_\gamma(x)$ is convex on Hadamard manifolds. By the triangle inequality, $b_{\gamma}$ is 1-Lipschitz. In fact, $\| \nabla b_\gamma \| = 1 $ by the Gauss Lemma. In particular, $b_\gamma$ is differentiable a.e. and, on a Hadamard manifold, it is of class $\mathcal{C}^2$ \cite{heintze1977geometry, Ballmann1995LecturesOS, bessa2015curvature}.

\end{proof}

We now provide a proof of \Cref{trm:product-general} after restating it.
\begin{theorem*}
     Let $(\M,g)$ be a Hadamard manifold and $\N = \M^m$ the product manifold equipped with the product metric. Let $a:\R^m\to\R^m$ be Lipschitz continuous as a vector-valued function. Assume
    \[
        a_k(z)\geq0,
        \qquad k=1,\dots,m,
        \qquad z\in\R^m,
    \]
    and assume that $R(z):=z-a(z)$ is nonexpansive in the Euclidean norm on $\R^m$.
    Define
    \[
         T:\N\to\N,
         \qquad
         T(X)^k=\Phi_{a_k(B(X))}(x^k),
         \qquad k=1,\dots,m.
     \]
    Then $T$ is nonexpansive on $\N$.
\end{theorem*}
\begin{proof}
Let $X,Y\in\mathcal N$, and let $\Gamma:[0,1]\to\mathcal N$ be the
minimizing geodesic joining them. Write
\[
    \Gamma(s)=(\gamma^1(s),\dots,\gamma^m(s)).
\]
It is enough to show
\[
    L(T\circ\Gamma)\le L(\Gamma).
\]

For each $k=1,\dots,m$, set
\[
    r_k(s):=b(\gamma^k(s)),
    \qquad
    r(s):=B(\Gamma(s))=(r_1(s),\dots,r_m(s))\in\mathbb R^m.
\]
Since $\Gamma$ is a smooth geodesic and $b\in\mathcal C^2$, the curve
$r$ is $\mathcal C^1$, hence Lipschitz on $[0,1]$. Since $R:=I-a$ is
nonexpansive, $R\circ r$ is Lipschitz. Therefore $R\circ r$ is
differentiable for a.e. $s\in[0,1]$, and at such points
\[
    \|(R\circ r)'(s)\|_2
    =
    \lim_{h\to0}
    \frac{\|R(r(s+h))-R(r(s))\|_2}{|h|}
    \le
    \lim_{h\to0}
    \frac{\|r(s+h)-r(s)\|_2}{|h|}
    =
    \|r'(s)\|_2.
\]
Moreover, since $a=I-R$, the curve $a\circ r$ is Lipschitz as well.
Hence $a\circ r$ is differentiable for a.e. $s$, and at points where
both derivatives exist,
\[
    (R\circ r)'(s)=r'(s)-(a\circ r)'(s).
\]

For a.e. $s$, decompose each component velocity as
\[
    \dot\gamma^k(s)
    =
    r_k'(s)\grad b(\gamma^k(s))+\xi_H^k(s),
    \qquad
    \xi_H^k(s)\perp \grad b(\gamma^k(s)).
\]
Since $\|\grad b\|=1$, this gives
\[
    \|\dot\Gamma(s)\|^2
    =
    \sum_{k=1}^m\|\dot\gamma^k(s)\|^2
    =
    \|r'(s)\|_2^2+\sum_{k=1}^m\|\xi_H^k(s)\|^2.
\]

Now set
\[
    A_k(s):=a_k(r(s))=a_k(B(\Gamma(s))).
\]
Then
\[
    T(\Gamma(s))^k
    =
    \Phi_{A_k(s)}(\gamma^k(s)).
\]
At a.e. $s$, using $\partial_t\Phi_t=-\grad b\circ\Phi_t$, we obtain
\[
    \frac{d}{ds}T(\Gamma(s))^k
    =
    D\Phi_{A_k(s)}(\dot\gamma^k(s))
    -
    A_k'(s)\grad b(T(\Gamma(s))^k).
\]
Since
\[
    D\Phi_t(\grad b)=\grad b\circ\Phi_t,
\]
we can rewrite this as
\[
    \frac{d}{ds}T(\Gamma(s))^k
    =
    D\Phi_{A_k(s)}
    \left(
        \dot\gamma^k(s)-A_k'(s)\grad b(\gamma^k(s))
    \right).
\]
Using the decomposition
\[
    \dot\gamma^k(s)
    =
    r_k'(s)\grad b(\gamma^k(s))+\xi_H^k(s)
\]
and the identity
\[
    r_k'(s)-A_k'(s)=(R\circ r)'_k(s),
\]
we get
\[
    \frac{d}{ds}T(\Gamma(s))^k
    =
    D\Phi_{A_k(s)}
    \left(
        \xi_H^k(s)+(R\circ r)'_k(s)\grad b(\gamma^k(s))
    \right).
\]
Since $A_k(s)=a_k(r(s))\ge0$, the differential of the Busemann flow
$\Phi_{A_k(s)}$ is nonexpansive. Therefore
\[
    \left\|\frac{d}{ds}T(\Gamma(s))^k\right\|^2
    \le
    \left\|
        \xi_H^k(s)+(R\circ r)'_k(s)\grad b(\gamma^k(s))
    \right\|^2.
\]
Because $\xi_H^k(s)\perp\grad b(\gamma^k(s))$ and $\|\grad b\|=1$, this gives
\[
    \left\|\frac{d}{ds}T(\Gamma(s))^k\right\|^2
    \le
    \|\xi_H^k(s)\|^2+\left|(R\circ r)'_k(s)\right|^2.
\]
Summing over $k$ gives
\[
    \left\|\frac{d}{ds}T(\Gamma(s))\right\|^2
    \le
    \sum_{k=1}^m\|\xi_H^k(s)\|^2+\|(R\circ r)'(s)\|_2^2.
\]
Using $\|(R\circ r)'(s)\|_2\le\|r'(s)\|_2$, we obtain
\[
    \left\|\frac{d}{ds}T(\Gamma(s))\right\|^2
    \le
    \sum_{k=1}^m\|\xi_H^k(s)\|^2+\|r'(s)\|_2^2
    =
    \|\dot\Gamma(s)\|^2.
\]
Thus
\[
    L(T\circ\Gamma)
    =
    \int_0^1\left\|\frac{d}{ds}T(\Gamma(s))\right\|\,ds
    \le
    \int_0^1\|\dot\Gamma(s)\|\,ds
    =
    L(\Gamma).
\]
Since $\Gamma$ is minimizing,
\[
    d_{\mathcal N}(T(X),T(Y))
    \le
    L(T\circ\Gamma)
    \le
    L(\Gamma)
    =
    d_{\mathcal N}(X,Y).
\]
Therefore $T$ is nonexpansive.
\end{proof}
\section{Calculations for tightness of the nonexpansiveness bound}\label{app:calculations-tight-bound}
We now provide an explicit example on $\mathbb{D}^2$ demonstrating that the bound in \eqref{eq: condition for nonexpansiveness} is tight.

Fix $P\in S^1$, i.e. $P\in\mathbb{R}^2$ with $\|P\|_2=1$. Let us consider two points
\[
P_i=-r_iP,\qquad 0<r_i<1,\,\,i=1,2,
\]
and
\[
b_P(x)=\log\left(\frac{\|x-P\|^2}{1-\|x\|^2}\right).
\]
On the diameter through $P$, write any point as $x=\rho P$, where $\rho\in (-1,1)$. Then
\[
b_P(\rho P) = \log\left(\frac{(1-\rho)^2}{1-\rho^2}\right) = \log\left(\frac{1-\rho}{1+\rho}\right) = -2\tanh^{-1}(\rho).
\]
For $P_i=-r_iP$, this gives
\[
b_P(P_i) = b_P(-r_iP) = 2\tanh^{-1}(r_i)>0.
\]
We then consider the potential
\begin{equation}\label{eq:AppPotential}
V(x)=\frac12\operatorname{ReLU}(b_P(x))^2,
\end{equation}
i.e., set $\beta=0$ and $\lambda=1$. The update it defines is
\[
T_\tau(x) = \Phi_{\tau\operatorname{ReLU}(b_P(x))}(x).
\]
At $P_1$ and $P_2$, $\mathrm{ReLU}$ is active, so the flow time is
\[
s_i(\tau) = \tau b_P(P_i) = 2\tau\tanh^{-1}(r_i).
\]
Let us now use the defining property of the Busemann flow $b_P(\Phi_s(x))=b_P(x)-s$.
Let $\gamma_i(\tau):=T_\tau(P_i)$. Since the flow stays on the diameter, we write $\gamma_i(\tau)=\rho_i(\tau)P$. Then
\[
b_P(\gamma_i(\tau)) = b_P(P_i)-s_i(\tau) = 2\tanh^{-1}(r_i)-2\tau\tanh^{-1}(r_i) = 2(1-\tau)\tanh^{-1}(r_i).
\]
But also $b_P(\rho_i(\tau)P) = -2\tanh^{-1}(\rho_i(\tau))$. Therefore $\rho_i(\tau) = \tanh\left((\tau-1)\tanh^{-1}(r_i)\right)$. We conclude that
\[
\gamma_i(\tau) = T_\tau(P_i) = \tanh\left((\tau-1)\tanh^{-1}(r_i)\right)P.
\]
The expansivity ratio to compute is now $\ell(\tau):=d(\gamma_1(\tau),\gamma_2(\tau))/d(P_1,P_2)$. Since both Riemannian distances are between points of the form $\lambda_1 P$ and $\lambda_2 P$,  $\lambda_1,\lambda_2\in (-1,1)$, we remark that
\[
d(\lambda_1P,\lambda_2 P) = 2\tanh^{-1}\left(\left\|\frac{\lambda_1P-\lambda_2P}{1-\lambda_1\lambda_2\|P\|^2}\right\|\right).
\]
Since $\|P\|_2=1$, we recover that 
\[
d(\lambda_1P,\lambda_2P)=2\left|\tanh^{-1}\left(\frac{\lambda_1-\lambda_2}{1-\lambda_1\lambda_2}\right)\right|=2|\tanh^{-1}(\lambda_1)-\tanh^{-1}(\lambda_2)|.
\]
This allows us to recover
\[
d(\gamma_1(\tau),\gamma_2(\tau))=2|\tau-1|\cdot|\tanh^{-1}(r_1)-\tanh^{-1}(r_2)|=|1-\tau|d(P_1,P_2),
\]
and that $\ell(\tau)=|1-\tau|$. Therefore, $\ell(\tau)\leq 1$ if and only if $0\leq \tau\leq \tau_{\max}=2$, which is the theoretical bound predicted by \eqref{eq: condition for nonexpansiveness} for the potential in \eqref{eq:AppPotential}.

\end{document}